\documentclass[11pt]{article}%
\usepackage{float} % fix table inside context 
\usepackage{fullpage}
\usepackage{amsfonts}
\usepackage{amsmath}
\usepackage{amssymb}
\usepackage{amsbsy}
\usepackage{amsthm}
\usepackage{mathtools}
\usepackage{epsfig}
\usepackage{graphicx}
\usepackage{colordvi}
\usepackage{graphics}
\usepackage{color}
\usepackage{bm}
\usepackage{verbatim} 
\numberwithin{equation}{section}
\usepackage{tikz}
\usepackage{ulem}

\newtheorem{thm}{Theorem}[section]
\newtheorem{lemma}[thm]{Lemma}

\newtheorem{alg}[thm]{Algorithm}
\newtheorem{remark}[thm]{Remark}
\newtheorem{assumption}[thm]{Assumption}

\newcommand{\nr}[1]{\ensuremath{\left\|{#1} \right\|}}

\usepackage{algpseudocode}

\newcommand{\be}{\begin{equation}}
\newcommand{\ee}{\end{equation}}
\newcommand{\bea}{\begin{eqnarray}}
\newcommand{\eea}{\end{eqnarray}}
\newcommand{\beas}{\begin{eqnarray*}}
	\newcommand{\eeas}{\end{eqnarray*}}

\newcommand{\vertiii}[1]{{\left\vert\kern-0.25ex\left\vert\kern-0.25ex\left\vert #1
		\right\vert\kern-0.25ex\right\vert\kern-0.25ex\right\vert}}

\newcommand{\normiii}[1]{{\left\vert\kern-0.25ex\left\vert\kern-0.25ex\left\vert #1
		\right\vert\kern-0.25ex\right\vert\kern-0.25ex\right\vert}}

\usepackage{enumerate}
\usepackage{multirow}
\date{}
\begin{document}
\title{A simple-to-implement nonlinear preconditioning of Newton’s
method for solving the steady Navier-Stokes equations}

\author{ 	
Muhammad Mohebujjaman\thanks{Department of Mathematics, University of Alabama at Birmingham, AL 35294, USA, mmohebuj@uab.edu; Partially supported by the National Science Foundation grant DMS-2425308.}
\and
Mengying Xiao\thanks{Department of Mathematics and Statistics, University of West Florida, Pensacola, FL 32514, mxiao@uwf.edu.}
\and
Cheng Zhang \thanks{Department of Mechanical Engineering, University of West Florida, Pensacola, FL 32514, czhang@uwf.edu.}
	}

	\maketitle
	
	\begin{abstract}{
	Newton’s method for solving stationary Navier-Stokes equations (NSE) is known for its fast convergence; however, it may fail when provided with a poor initial guess. This work presents a simple-to-implement nonlinear preconditioning technique for Newton’s iteration that retains quadratic convergence and expands the domain of convergence. The proposed AAPicard-Newton method adds an Anderson accelerated Picard step at each iteration of Newton’s method for solving NSE. This approach has been shown to be globally stable with a relaxation parameter $\beta_{k+1} \equiv 1$ in the Anderson acceleration optimization step, converges quadratically, and achieves faster convergence with a small convergence rate for large Reynolds numbers. Several benchmark numerical tests have been carried out and are well aligned with the theoretical results. 
		}
\end{abstract}



	\section{Introduction}
	
	The incompressible Navier-Stokes equations (NSE) are governed by  the form
	\begin{equation}\label{NS}
		\left\{\begin{aligned}
			-\nu \Delta u+u\cdot\nabla u+ \nabla p&={f} \quad \text{in}~\Omega,\\
			\nabla\cdot {u}&=0\quad \text{in}~\Omega,
		\end{aligned}\right.
	\end{equation}
	on a domain $\Omega \subset \mathbb{R}^d, d=2,3$. Here, $u$ is the velocity of the fluid, $p$ is the pressure, $\nu$ is the kinematic viscosity of the fluid,  and ${f}$ is an external forcing term. The parameter $Re:=\frac{1}{\nu}$ represents the Reynolds number, which characterizes the complexity of the fluid problem.  This work focuses on nonlinear solvers for the steady system defined by \eqref{NS} with homogeneous Dirichlet boundary conditions; however, with some extra work, the results can also be extended to solve the time-dependent NSE at a fixed temporal discretization step, as well as non-homogeneous mixed Dirichlet/Neumann boundary conditions. 
	
Finding the NSE solution efficiently and accurately is an open problem, as the NSE is
notorious for being one of the hard nonlinear differential equations for which analytical solutions are often unavailable. In practice, the nonlinear problem may exhibit poor convergence
rates or be sensitive to the choice of initial conditions, as seen with methods like Picard and Newton methods. The Picard method is known for its global stability if a unique solution exists. Specifically, it guarantees convergence to a solution regardless of the initial guess. Moreover, for sufficiently small problem data, the Picard is globally convergent with a linear convergence rate smaller than 1 \cite{GR86, PRX19}. However, this can be seen as a limitation in practical applications due to the convergence rate being close to 1. On the other hand, Newton's iteration demonstrates local quadratic convergence for small data and a sufficiently close initial guess \cite{GR86, LHRV23}. 
Common alternative strategies involve using  methods with a larger domain of convergence initially (such as Picard, Anderson accelerated Picard methods \cite{EPRX20, PRX19} or damped Newton \cite{NWbook06}) and subsequently switching to Newton's method once the iterates are sufficiently  close to the solution \cite{John16}.

Nonlinear preconditioning is a widely adopted technique in many fields \cite{G01, NSK16, YHC16}, including computational fluid dynamics, structural analysis, and optimization problems, where nonlinearities are prevalent. This technique enhances the efficiency and robustness of numerical methods for solving nonlinear problems by converting the nonlinear problem into better-conditioned forms. Several works \cite{CK02, DGKKM16, LKK18} have  focused on
nonlinear preconditioning for Newton's method, leading to faster convergence and a larger domain of convergence. For instance, \cite{PRTX24} proposes a nonlinear preconditioning of Newton's method by adding a Picard step at each iteration of Newton's method, called the
Picard-Newton method (P-N), and is in the form of\\

%\begin{itemize}
\noindent Step 1. Finding $(\tilde u_{k+1} , \tilde p_{k+1})$ such that
\begin{align*}
-\nu \Delta \tilde u_{k+1}+u_k\cdot\nabla \tilde u_{k+1}+ \nabla \tilde p_{k+1}&={f}, \\
		\nabla\cdot {\tilde u}_{k+1}&=0. 
	\end{align*}
Step 2. Finding $(u_{k+1}, p_{k+1})$ such that
\begin{align*}
	-\nu \Delta u_{k+1}+ \tilde u_{k+1}\cdot\nabla u_{k+1}+ u_{k+1}\cdot\nabla \tilde u_{k+1} + \nabla p_{k+1} &={f} + \tilde u_{k+1} \cdot\nabla \tilde u_{k+1}, \\
		\nabla\cdot {u}_{k+1}&=0.
\end{align*}
%\end{itemize}
We denote it by $u_{k+1} = g_{PN}(u_k) = g_N(g_P(u_k))$.  In Step 1, we apply the standard Picard iteration, using the previous velocity $u_k$. In Step 2, we utilize the usual Newton's method with the intermediate velocity $\tilde u_{k+1}.$ Thus, the P-N method linearizes the nonlinear system \eqref{NS}, though it requires two linear solves at each iteration. The P-N method is advantageous because it is easy to implement, globally stable, quadratically convergent, and has an extended convergence radius compared to the usual Newton method. The 2D lid-driven cavity numerical experiment in \cite{PRTX24} indicates that this 2-step method can converge  for significantly higher Reynolds numbers ($Re \le 12000$) than the  Picard method ($Re \le 4000$), Newton's method ($Re \le 2500$), or Newton with line search ($Re \le 3000$). This technique has been applied to other nonlinear partial differential equations; see \cite{FHHR25}.

Our work is inspired by \cite{PRTX24} and aims to further improve the performance of the Newton method by adding an Anderson accelerated Picard (AAPicard) preconditioning step \cite{Anderson65, EPRX20}. Anderson acceleration (AA) adds an easy-to-implement optimization step of Picard solutions in the iteration, leading to solid convergence results from recent work \cite{EPRX20, PR21, RX23}. AA improves the convergence of linearly convergent methods by slightly reducing the convergence rate and adjusting the convergence order of sublinearly and/or superlinearly convergent methods from $r$ to $\frac{r+1}2$. Specifically, for the NSE, \cite{PRX19, X23} found that AA enhances the Picard iteration with a small decrease in the convergence rate while decelerating Newton's method to a superlinear convergence order of $1.5$. Thus, it is reasonable to apply AA solely to the Picard step to obtain optimal performance, and we refer to this approach as the AAPicard-Newton method. We establish that the AAPicard-Newton method is globally stable, quadratically convergent,
and optimizing convergent performance when $\beta_{k+1}\equiv 1$. Our theoretical results, along with numerical tests, indicate that the AAPicard-Newton method significantly improves the convergence for large Reynolds numbers. For example, the 2D lid-driven cavity solved by AAPicard-Newton method converges for Reynolds number $Re \le 20,000$ and the 3D lid-driven cavity test for $ Re\le 3000$. These values are higher than $Re \le 12,000$ and $Re\le 1800$, respectively, for the results obtained using the P-N method.

The rest of the paper is arranged as follows. Section 2 provides mathematical preliminaries for a smoother analysis to follow. Section 3 analyzes the AAPicard-Newton method with various depths $m$ and relaxation parameter $\beta$. Section 4 presents several benchmark numerical experiments to verify the theoretical results. Section 5 discusses the conclusion and future study.

\section{Mathematical preliminaries}

We denote the natural function spaces for the NSE by 
	\begin{align}
		&Q:=\{v\in {L}^2(\Omega): \int_{\Omega}v\ dx=0\},\\
		& X:=
		\{v\in H^1\left(\Omega\right): v=0~~\text{on}~ \partial\Omega\},\\
		& V:=
		\{v\in X:  (\nabla \cdot v,q)=0\ \forall q\in Q\},
	\end{align}
	where $\Omega$ is an open connected set. The $L^2$ inner product and norm are denoted by $(\cdot,\cdot)$ and $\|\cdot\|$, respectively.  The notation $\langle \cdot,\cdot\rangle$,  defined as $\langle f,v\rangle = \int_\Omega f v\ dx,$ is used to represent the duality between $H^{-1}$ and $X$, and $\|\cdot\|_{-1}$ denotes the norm on $H^{-1}$.
		
 We define the nonlinear term: for all $v,w,z\in X,$
\[
b^*(v,w,z) = (v\cdot\nabla w,z) + \frac12 ((\nabla \cdot v)w,z).
\]
Alternative energy-conserving formulations of the nonlinear term, including the rotational form and the EMAC (Energy, Momentum, and Angular Momentum Conserving) scheme \cite{CHOR17,GS98,OR20}, can be utilized in this framework. A key characteristic of $ b^*$ is its skew-symmetric property \cite{temam}, which is defined by the equation:
\begin{align}
b^*(u,v,v) = 0,
\label{bsym}
\end{align}
for all $ u, v \in X $. If the first argument of $ b^* $ meets the condition $ \|\nabla \cdot u\| = 0 $, then the skew-symmetry is not required and does not influence the outcome. However, when certain finite element subspaces of $ X $ and $ Q $ are employed—like Taylor-Hood elements—it's possible that the discretely divergence free space does not produce entirely divergence-free functions. Therefore, to ensure generality and address all scenarios, we incorporate skew-symmetry in our analysis.
 The following inequality holds for all $v,w,z\in X$
\begin{align}
b^*(v,w,z)&\leq M \|\nabla v\|\|\nabla w\|\|\nabla z\|, \label{bstarbound}
	\end{align}
where $M$ dependent only on the domain $\Omega$, see \cite{laytonBook,temam}.

We present two inequalities in a Hilbert space $X$. For any vectors $u,v\in X$, we have the polarization identity \cite{S96} 
 \begin{align}
\label{eqn:polar}
2(u,v) &= \|u\|^2 + \|v\|^2 -\|u-v\|^2, \end{align}
and Young's inequality \cite{laytonBook}
\begin{align}
\label{eqn:youn}
(u,v) &\le \frac{\epsilon}2 \|u\|^2 + \frac{1}{2\epsilon} \|v\|^2,
\end{align}
for any $\epsilon>0.$

\subsection{NSE preliminaries}

Here we present the weak form \cite{temam} of the NSE \eqref{NS}, given by: Find $u\in V$ satisfying 
\begin{equation}
\nu(\nabla u,\nabla v) + b^*(u,u,v) = \langle f,v\rangle, \ \forall v\in V. \label{weakNS}
\end{equation}

It is well known that for any $f\in H^{-1}(\Omega)$ and $\nu>0$, the weak steady NSE system \eqref{weakNS} is well-posed \cite{GR86,laytonBook} if  the small data condition is satisfied \[
\sigma:=M\nu^{-2} \|f \|_{-1}<1,
\] and that any solution to \eqref{NS} or  \eqref{weakNS} satisfies
\begin{equation}
\| \nabla u \| \le \nu^{-1} \| f \|_{-1}  
, \label{nsstab}
\end{equation}
 by setting $v =u$ in equation \eqref{weakNS} and applying \eqref{bsym} and H\"older's inequality. 
For the rest of this paper, we assume $\sigma <1$. Of course, all works stated below are globally applicable and can be extended to $\sigma >0$ for local results with extra work (see \cite{laytonBook} Chapter 6.4, \cite{PRTX24}).

 The newly introduced, simple-to-implement, 2-step iterative method  -- Picard-Newton method \cite{PRTX24} is stated below.
\begin{alg}[Picard-Newton method \cite{PRTX24}]
\label{alg:pn}
The Picard-Newton method consists of applying the composition of the Newton and  Picard iteration for solving  Navier-Stokes equations: $g_{N}\circ g_{P}$, i.e.,
\begin{enumerate}[Step 1:]
\item[Step 1:] Find $\hat u_{k+1} = g_P(u_{k})$ by finding $\hat  u_{k+1}\in V$ 
satisfying  for all $v\in V$
\begin{equation}
\nu(\nabla \hat u_{k+1},\nabla v) + b^*(u_k,\hat u_{k+1},v) = \langle f,v \rangle.
 \label{wp}
\end{equation}
\item[Step 2:]  Find $u_{k+1} = g_N(\hat u_{k+1})$  by finding $u_{k+1}\in V$ satisfying  for all $v\in V$
\begin{align}
\label{wn}
\nu(\nabla u_{k+1}, \nabla v) +b^*( \hat u_{k+1}, u_{k+1}, v) + b^*(u_{k+1} , \hat u_{k+1}, v) - b^*(\hat u_{k+1}, \hat u_{k+1}, v) =  \langle f,v \rangle. 
\end{align}
\end{enumerate}
\end{alg}
\cite{PRTX24} has manifested that Algorithm \ref{alg:pn} is globally stable,
 \begin{align}
\label{pnstable}
 \|\nabla \hat u_k\|\le \nu^{-1}\|f\|_{-1}, \quad  \|\nabla  u_{k+1}\| \le & \frac{1+\sigma}{1-\sigma}\nu^{-1}\|f\|_{-1}.  \end{align}
  quadratically convergent for $\sigma<1$, and has a larger domain of convergence than the usual Newton's method.  Similar local convergence results can be derived for $\sigma \ge 1$ and any given data $f$, see \cite{PRTX24}.

%{\color{red} Xiao says: The bound of $u_{k}$ is incorrect in [18].
%\begin{proof}
%Let $v = \hat u_{k+1}$ in \eqref{wp} eliminates the second term and applying Cauchy-Schwarz inequality gives
%\begin{align*}
%\nu \| \nabla \hat u_{k+1}\| \le \|f\|_{-1}.
%\end{align*}
%Then let $v =u_{k+1}$ in \eqref{wn} eliminates the second term and using \eqref{bstarbound}, Cauchy-Schwarz inequalities yields
%\begin{align*}
%\nu\|\nabla u_{k+1}\|^2 \le \|f\|_{-1}\|\nabla u_{k+1}\| +M\|\nabla u_{k+1}\|^2 \|\nabla \hat u_{k+1}\| +M\|\nabla u_{k+1}\| \|\nabla \hat u_{k+1}\|^2,
%\end{align*}
%which reduces to 
%\begin{align*}
%\nu(1-\nu^{-2}M\|f\|_{-1}) \|\nabla u_{k+1}\| \le& \|f\|_{-1} +\nu^{-2}M \|f\|_{-1}^2, \\
%\|\nabla  u_{k+1}\| \le & \frac{1+\alpha}{1-\alpha}\nu^{-1}\|f\|_{-1}.
%\end{align*}
%We complete the proof.
%\end{proof}
%}

\subsection{Anderson acceleration preliminary}

Anderson acceleration has recently been shown to enhance the convergence of linearly converging fixed point methods, such as the Picard method for NSE \cite{EPRX20,PR21,PRX19}, and reduce the asymptotic convergence order of the Newton's method for NSE \cite{RX23, X23}. It is optimal to apply AA solely to the Picard step of the Picard-Newton method \cite{PRTX24}.

Given a fixed point function $g:X\rightarrow X$ with $X$ a Hilbert space with norm $\| \cdot \|_X$,
the Anderson acceleration algorithm with depth $m =1,2,3,\dots$ and damping parameters $0 < \beta_{k+1} \le 1$ is given by:
\\ \ \\
Step 0: Choose $x_0\in X.$\\
Step 1: Find $w_1\in X $ such that $w_1 = g(x_0)-x_0$.  
Set $x_1 = x_0 + w_1$. \\
Step $k+1$: For $k=1,2,3,\ldots$ Set $m_k = \min\{ k, m\}.$\\
\indent [a.] Find $w_{k+1} = g(x_k)-x_k$. \\
\indent [b.] Solve the minimization problem for $\{ \alpha_{j}^{k+1}\}_{k-m_k}^k$
\begin{align}\label{eqn:opt-v0}
\min_{\sum_{j=k-m_k}^{k} \alpha_j^{k+1}  = 1} 
\left\| \sum_{j=k-m_k}^{k} \alpha_j^{k+1} w_{j+1} \right\|_X
\end{align}
\indent [c.] For a selected damping factor $0 < \beta_{k+1} \le 1$, set
\begin{align}\label{eqn:update-v0}
x_{k+1} =  \sum_{j= k-m_k}^k \alpha_j^{k+1} x_{j}
 + \beta_{k+1} \sum_{j= k-m_k}^k \alpha_j^{k+1} w_{j+1},
\end{align}
where $w_{k+1}:= g(x_k)-x_k$ represents the stage $k$ residual.  
 
 \begin{remark}
 We assume the $\alpha_j^{k+1}$ are uniformly bounded.  As discussed in \cite{PR21,PR23}, this is equivalent to assuming the full column rank of the matrix with columns $(w_{j+1}-w_j)_{j=k,k-1,...k-m}$ and can be controlled by the length and angle filtering.
 \end{remark}

It is the optimization step that improves the linearly convergent methods, and we define the Anderson gain 
\begin{align}\label{thetadef}
\theta_k := \frac{ \left\| \sum_{j=k-m_k}^{k} \alpha_j^{k+1} w_{j+1} \right\|_X } { \| w_{k+1} \|_X }.
\end{align}
Clearly $0<\theta_k\le 1$. With this, it can be proven that AA improves the linear convergence rate by scaling it by the gain factor $\theta_k$ of the underlying AA optimization problem  \cite{EPRX20,PR21,PRX19} .  
For AA with depth $m$, using the result from Theorem 5.1 of \cite{PR21} we have that
\begin{align}\label{eqn:tgenm}
\nr{w_{k+1}}_X & \le \nr{w_k}_X \Bigg\{
 \theta_k ((1-\beta_{k}) + \kappa_g \beta_{k})
+ C \hat \kappa_g 
   \sum_{n = k-{m_{k-1}}}^{k} 
\nr{w_n}_X
  \Bigg\},
\end{align}
where $\kappa_g$ is the linear convergence rate of the usual fixed point iteration, $\hat \kappa_g$ is the Lipschitz constant of $g'$, $C$ depends on 
relaxation and gain parameters, as well as the degree to which the past $m$ differences $w_{j+1}-w_j$ are linearly independent.
The residual $w_{k+1}$ is bounded by a linear term and several quadratic terms. The linear term will be dominant if the initial guess is good enough, and therefore, AA improves the rate of convergence as $\theta_k<1$ and further enhances the performance if 
 $0<\beta_k < 1$, especially  for large $\kappa_g$. A greater depth results in a smaller $\theta_k$, and accelerates convergence. However, finding the optimal value of $\beta_k$ is case-dependent and should be evaluated through numerical tests.

\section{Analysis of AAPicard-Newton}
\cite{PRTX24} manifests that the Picard-Newton method converges quadratically (globally for $\sigma <1$ and locally for $\sigma>0$). In the following subsections, we will study the convergence behavior of AA with depth $m=1$ applied to this method, and present the general depth case $m= 2,3,\dots$ in the subsequent subsections. To make the work clean, we restrict ourselves to $\sigma<1$, but all results presented below can be extended locally to $\sigma>0$ with some additional effort (see \cite{laytonBook} Chapter 6.4, \cite{PRTX24}), and hence the analysis is omitted here.

\subsection{AAPicard-Newton $m=1$}
We now analyze the AAPicard-Newton method using depth $m=1$.  We begin by formally stating the method.
\begin{alg}[AAPicard-Newton $m=1$]
\label{alg:aapn1}
The AAPicard-Newton iteration with depth $m =1$ consists of applying the composition of the Newton and Anderson accelerated Picard iterations for solving Navier-Stokes equations: $ g_N \circ g_{AP} $, i.e.,
\begin{enumerate}
\item[Step 1:] Find $\tilde u_{k+1} = g_P(u_{k})$ by finding $\tilde  u_{k+1}\in V$ 
satisfying for all $v\in V$
\begin{equation}
\nu(\nabla \tilde u_{k+1},\nabla v) + b^*(u_k,\tilde u_{k+1},v) = \langle f,v \rangle.
 \label{wpaa}
\end{equation}
\item[Step 2:] 
For a selected damping factor $0< \beta_{k+1} \le 1$, set 
  \begin{align}
 \hat u_{k+1} =  & \beta_{k+1} \left( (1-\alpha_{k+1})  \tilde u_{k+1} + \alpha_{k+1} \tilde u_{k}  \right)
 + (1-\beta_{k+1}) \left(   (1-\alpha_{k+1})   u_{k} + \alpha_{k+1}  u_{k-1} \right) 
 \nonumber\\
 = & (1-\alpha_{k+1})  \tilde u_{k+1} + \alpha_{k+1} \tilde u_{k}  - (1-\beta_{k+1}) w_{k+1}^\alpha ,
   \label{hatu_dfn}
  \end{align}
where $\alpha_{k+1}$ minimizes 
\begin{align}
\label{eqn:min1}
 \|\nabla w_{k+1}^\alpha\| 
 \coloneqq 
 \left\|   (1-\alpha_{k+1}) \nabla (\tilde u_{k+1} - u_k) + \alpha_{k+1} \nabla (\tilde u_k - u_{k-1} ) \right\|.
\end{align}
\item[Step 3:] Find $u_{k+1} = g_N(\hat u_{k+1})$  by finding $u_{k+1}\in V$ satisfying  for all $v\in V$
\begin{align}
\label{wnaa}
\nu(\nabla u_{k+1}, \nabla v) +b^*( \hat u_{k+1}, u_{k+1}, v) + b^*(u_{k+1} , \hat u_{k+1}, v) - b^*(\hat u_{k+1}, \hat u_{k+1}, v) =  \langle f,v \rangle. 
\end{align}
\end{enumerate}
\end{alg}
\begin{remark}
The AAPicard-Newton $m=1$ Algorithm in \cite{PRTX24,FHHR25} differs slightly from Algorithm \ref{alg:aapn1} discussed here. For comparison, we present the algorithm in \cite{PRTX24} below. \\  
{\bf Algorithm} (PRTX \cite{PRTX24} AAPicard-Newton algorithm $m=1$ and no relaxation):\\
The method consists four steps at each iteration, and is described as follows:
 \begin{enumerate}
 \item[Step 1.] Find $\tilde u_{k+1} = g_P(u_k)$ by finding $\tilde u_{k+1}\in V$ satisfying \eqref{wpaa} for all $v\in V$.
\item[Step 2.] Find $\tilde {\tilde u}_{k+1} = g_P(\tilde u_{k+1})$ by finding $\tilde {\tilde u}_{k+1}\in V$ satisfying for all $v\in V$
\begin{align*}
\nu(\nabla \tilde{\tilde  u}_{k+1},\nabla v) + b^*(\tilde u_{k+1},\tilde{\tilde  u}_{k+1},v) = \langle f,v \rangle.
\end{align*}
\item[Step 3.] Set
 \begin{align*}
 \hat u_{k+1} =  &(1-\alpha_{k+1})  \tilde{\tilde u}_{k+1} + \alpha_{k+1} \tilde u_{k+1} ,
  \end{align*}
where $\alpha_{k+1}$ minimizes 
\begin{align*}
 \left\|   (1-\alpha_{k+1}) \nabla (\tilde {\tilde u}_{k+1} - \tilde u_{k+1} ) + \alpha_{k+1} \nabla (\tilde u_{k+1} - u_k)  \right\|.
\end{align*}
\item[Step 4.] Find $u_{k+1} = g_N(\hat u_{k+1})$  by finding $u_{k+1}\in V$ satisfying \eqref{wnaa} for all $v\in V$.
 \end{enumerate}
The differences between the two algorithms lie in Step 2 and Step 3. In Step 3, the algorithm uses Picard solutions from the current iteration only, requiring an additional Picard solve in Step 2 at each iteration. Thus, this method takes four linear solves per iteration: two Picard solves, one linear solve for optimization, and one Newton solve. Generally, increasing the Anderson depth $m$ enhances the convergence if the initial guess is good enough. However, adding more residual terms in Step 3 as $m$ increases will necessitate more Picard solves at each iteration. Consequently, for large $m$, the algorithm will need $m+3$ linear solves (including $m+1$ Picard solves, one linear solve for optimization, and one Newton solve) at each iteration. Furthermore, the computation time for Step 3 may increase significantly with larger $m$, making this algorithm less advisable for high depths.
 
 In contrast,  Algorithm \ref{alg:aapn1} involves only 1 Picard solve \eqref{wpaa} at each iteration and reuses the Picard solution(s) from the previous iteration(s) for the Anderson optimization step  \eqref{eqn:min1}. This same approach is taken by Algorithm \ref{alg:aapn2} and Algorithm \ref{alg:aapnm}, with depth $m\ge 2$. Overall, Algorithms \ref{alg:aapn1}, \ref{alg:aapn2} and \ref{alg:aapnm} require 3 linear solves at each iteration for any $m\ge 1$. This modification makes the method more practical for larger Anderson depth $m$, although it complicates convergence analysis.
\end{remark}

Let $\alpha_{k+1}$ minimize \eqref{eqn:min1}, define the Anderson gain
\begin{align}
\label{eqn:theta1}
\theta_{k+1} 
= \frac{\|\nabla w_{k+1}^\alpha\| }{\|\nabla (\tilde u_{k+1} - u_k)\|}
 =\frac{ \left\|   (1-\alpha_{k+1}) \nabla (\tilde u_{k+1} - u_k) + \alpha_{k+1} \nabla (\tilde u_k - u_{k-1} ) \right\| }{\|\nabla (\tilde u_{k+1} - u_k)\|}.
\end{align} 
Obviously, $0 \le \theta_{k+1} \le 1$, and moreover $\theta_{k+1} =1$ if and only if $\alpha_{k+1} =0$, which implies Algorithm \ref{alg:aapn1} is back to Algorithm \ref{alg:pn}.
From now on, we analyze this AAPicard-Newton iteration to show how it improves on Picard-Newton with $\theta_{k+1} <1$.  Because we are using AA and will draw from the AA theory of \cite{PR21}, convergence is analyzed in terms of residuals, not the difference between successive iterations. 
First, we assume on the parameters $\{\alpha_k\}$ we obtained from the Algorithm \ref{alg:aapn1}.
\begin{assumption}[$m=1$]
\label{ass:aapn1}
Let the sequence $\{\alpha_k\}$ from Algorithm \ref{alg:aapn1} be uniformly bounded such that for all $k$
\begin{align}
\label{eqn:aa1}
|\alpha_k| \le |1-\alpha_k| + |\alpha_k| \le C_A,
\end{align}
for some constant $C_A >0$.
\end{assumption}

Below presents the bounds for the Picard solution $\tilde u_{k+1}$ and the difference between successive iterations $\tilde u_{k+1} - \tilde u_k$, which is satisfied for any $\sigma >0$. Of course, it is also satisfied under the assumption $\sigma <1$.
\begin{lemma}%[$ \sigma >0$]
\label{lemma:pic}
For any positive integer $k$, we have
\begin{align}
\|\nabla \tilde u_{k+1}\| & \le \nu^{-1} \|f\|_{-1}, 
\label{ukbd1}\\
\label{eqn:picerr}
\|\nabla (\tilde u_{k+1} - \tilde u_{k})\| &\le  \sigma\|\nabla (u_k - u_{k-1})\|.
\end{align}
\end{lemma}
\begin{proof}
Setting $v = \tilde u_{k+1}$ in \eqref{wpaa} eliminates the second term and yields \eqref{ukbd1} using Cauchy-Schwarz inequality.

We next prove \eqref{eqn:picerr}. Subtracting \eqref{wpaa} with $k$ from \eqref{wpaa} with $k+1$ yields
\begin{align}
\label{eqn:errpic}
\nu (\nabla (\tilde u_{k+1} - \tilde u_k), \nabla v) + b^*(u_k - u_{k-1},\tilde u_k, v)+b^*(u_{k}, \tilde u_{k+1} - \tilde u_k, v) = 0.
\end{align}
Letting $v = \tilde u_{k+1} - \tilde u_k$ eliminates the last term and gives
\begin{align*}
 \|\nabla (\tilde u_{k+1} - \tilde u_k)\| \le \nu^{-1}M\|\nabla (u_k - u_{k-1})\| \|\nabla \tilde u_k\| \le \sigma \|\nabla (u_k - u_{k-1})\|,
\end{align*}
thanks to \eqref{bstarbound} and \eqref{ukbd1}.
\end{proof}

Next, we find the bounds of $\hat u_{k+1}$ and $u_{k+1}$. 
\begin{assumption}
\label{ass:ukbd}
Let $u_0$ be a good initial guess such that for all $k$, the inequality is satisfied
 \begin{align}
\|\nabla u_{k+1}\| \le L\nu^{-1}\|f\|_{-1},
 \label{ukbd2_beta}
\end{align}
for some constant $L>0$.
\end{assumption}

Consequently, applying the triangle inequality, \eqref{eqn:aa1} and \eqref{ukbd1} to equation \eqref{hatu_dfn}, we can bound $\hat u_{k+1}$ as 
\begin{align}
\label{aaukbd}
 \|\nabla \hat u_{k+1}\| 
&\le |1-\alpha_{k+1}| \|\nabla \tilde u_{k+1}\| + |\alpha_{k+1} | \|\nabla \tilde u_{k}\| 
+ |1-\alpha_{k+1}| \|\nabla  u_{k}\| + |\alpha_{k+1} | \|\nabla u_{k-1}\| 
\nonumber\\ &\le
 \max\{1,L\}C_A\nu^{-1} \|f\|_{-1},
\end{align}
thanks to $0< \beta_{k} \le 1.$
Moreover, 
from equation \eqref{hatu_dfn}, we have
\begin{align}
 \|\nabla (\hat u_{k+1} - \hat u_{k})\| & \le 
\|\nabla ( \tilde u_{k+1} - \hat u_{k})\| 
+ |\alpha_{k+1}| \|\nabla (\tilde u_{k+1} - \tilde u_{k})\| 
+ (1-\beta_{k+1})  \| \nabla w_{k+1}^\alpha\|
\nonumber \\
& \le 
\|\nabla ( \tilde u_{k+1} - \hat u_{k})\| 
+ \sigma C_A \|\nabla ( u_{k} -  u_{k-1})\| 
+ (1-\beta_{k+1})\theta_{k+1} \| \tilde u_{k+1} - u_{k}\|,
\label{eqn:haterr}
\end{align}
thanks to triangle inequality and \eqref{eqn:picerr}.

\begin{remark}
\label{stab1}
Assumption \ref{ass:ukbd} is not  necessary for $\beta_{k+1}=1$. As one can easily obtain 
$$\|\nabla \hat u_{k+1}\| \le C_A\nu^{-1}\|f\|_{-1},$$
from \eqref{hatu_dfn}, triangle inequality and \eqref{ukbd1}. And then we have the following equation similar to \eqref{pnstable}
% \textcolor{red}{(Jaman says: after the correction, they do not look similar. Need to adjust accordingly)}
$$ \|\nabla u_{k+1}\| \le  \frac{1+\sigma C_A^2}{1-\sigma C_A } \ \nu^{-1}\|f\|_{-1},
$$
by letting $v= u_{k+1}$ in equation \eqref{wnaa} and followed by applying \eqref{bstarbound} and H\"older's inequality. This indicates that Algorithm \ref{alg:aapn1} is globally stable when $\beta_{k+1}=1$ and $\sigma, \sigma C_A<1$.
\end{remark}

Next, we present two preliminary lemmas in order to show the convergence of Algorithm \ref{alg:aapn1}.  We first show that $\|\nabla (\tilde u_{k+1} - \hat u_k)\| = \mathcal{O}(\|\nabla (u_k -\hat u_k)\|)$, which means that the convergence of Algorithm \ref{alg:aapn1} is essentially determined by the Newton's iteration step. 

\begin{lemma}%[$ 0<\sigma<1$]
\label{lemma10}
For any integer $k$, we have 
\begin{align}
\|\nabla (\tilde u_{k+1} - u_k) \| &\le \nu^{-1} M \|\nabla (u_k - \hat u_k)\|^2. 
\label{eqn:quad1}
\end{align}
If Assumptions \ref{ass:aapn1}, \ref{ass:ukbd} and $\sigma C_A \max\{1,L\} <1$ hold, we have inequality
\begin{align}
\|\nabla (\tilde u_{k+1} - \hat u_k)\| &= \mathcal{O}(\|\nabla (u_k -\hat u_k)\|) = C_k \|\nabla (u_k -\hat u_k)\|,
 \label{eqn:equiv}
\end{align}
for some constants $C_k>0$ independent of $\nu, h$.
\end{lemma} 
\begin{proof}
Subtracting \eqref{wnaa} at iteration  $k$ from \eqref{wpaa} at iteration $k+1$, we obtain
\begin{align}
\label{wpn}
\nu(\nabla (\tilde u_{k+1} - u_k), \nabla v) + b^*( u_k, \tilde u_{k+1}-u_k,v) + b^*(u_k - \hat u_k,  u_{k} -\hat u_k,v)  = 0.
\end{align}
Setting $v= \tilde u_{k+1} - u_k$ eliminates the second term.
Applying \eqref{bstarbound} establishes \eqref{eqn:quad1}.

On the other hand, setting $v = \hat  u_k - u_k$ in equation \eqref{wpn} eliminates the 
third term. Applying \eqref{eqn:polar} to the first term and 
\eqref{bstarbound} then \eqref{eqn:youn} with $\epsilon = 4$ to the second term yields 
\begin{align*}
&\frac{\nu}2 \left( \|\nabla (\tilde u_{k+1} -u_k)\|^2 + \|\nabla ( u_k - \hat u_k)\|^2 - \|\nabla (\tilde u_{k+1} - \hat u_k )\|^2 \right)\\
= & b^*(u_k, \tilde u_{k+1} - \hat u_k, u_k - \hat u_k)\\
\le & \frac{\nu}4 \|\nabla (u_k- \hat u_k)\|^2 + \nu^{-1}M^2 \|\nabla u_k\|^2 \|\nabla (\tilde u_{k+1} -\hat u_k)\|^2.
\end{align*}
Dropping the term $\|\nabla (\tilde u_{k+1} - u_k)\|^2$ reduces to 
\begin{align}
\label{eqn:equi1}
\frac{\nu}4 \|\nabla ( u_k - \hat u_k)\|^2 &\le 
 \left( \frac{\nu}2+ \nu^{-1}M^2\|\nabla u_k\|^2 \right)\|\nabla (\tilde u_{k+1} - \hat u_k )\|^2 \nonumber\\
& 
\le
\nu \left( 
\frac{1}2+  L^2\sigma^2 \right) \|\nabla (\tilde u_{k+1} - \hat u_k )\|^2,
\end{align}
thanks to \eqref{ukbd2_beta}.

To get the reverse inequality necessary for \eqref{eqn:equiv}, we can
rearrange equation \eqref{wpn} to get 
\begin{align*}
\nu(\nabla (\tilde u_{k+1} - u_k), \nabla v) + b^*( \hat u_k, \tilde u_{k+1}-u_k,v) + b^*(u_k - \hat u_k, \tilde u_{k+1} -\hat u_k, v)  = 0.
\end{align*}
Setting $v = \tilde u_{k+1} - \hat u_k$ eliminates the third term. 
Again applying \eqref{eqn:polar} to the first term and \eqref{bstarbound}  followed by \eqref{eqn:youn} to the second yields
\begin{align*}
\frac{\nu}2 \left(  \|\nabla (\tilde u_{k+1} - u_k)\|^2 + \|\nabla (\tilde u_{k+1} - \hat u_k)\|^2 -\|\nabla (u_k - \hat u_k)\|^2  \right)\\
\le \frac{M}2 \|\nabla \hat u_k\| \left(  \|\nabla (\tilde u_{k+1} - u_k)\|^2 + \|\nabla (\tilde u_{k+1} - \hat u_k)\|^2 \right).
\end{align*}
This reduces to 
\begin{align*}
(1-\max\{1,L\}C_A\sigma) \left( \|\nabla (\tilde u_{k+1} - \hat u_k)\|^2  + \|\nabla (\tilde u_{k+1} -  u_k)\|^2 \right) 
 \le \|\nabla (u_k -\hat u_k)\|^2,
\end{align*}
using \eqref{aaukbd}.
Since $ \max\{1,L\}C_A\sigma <1$, dropping the term $\|\nabla (\tilde u_{k+1} -  u_k)\|^2$ yields
\begin{align}
 \label{eqn:equi2}
( 1- \max\{1,L\}C_A\sigma ) \|\nabla (\tilde u_{k+1} - \hat u_k)\|^2 \le \|\nabla (u_k -\hat u_k)\|^2.
\end{align}
Inequalities \eqref{eqn:equi1} and \eqref{eqn:equi2} show the equivalence of $ \tilde u_{k+1} - \hat u_k$ and $u_k -\hat u_k$ satisfying 
\begin{align*}
(1-\max\{1,L\}C_A\sigma )\|\nabla (\tilde u_{k+1} - \hat u_k)\| ^2 
\le \|\nabla (u_k - \hat u_k)\|^2 \le 
\left( 2+  4\sigma^2 L^2 \right) \|\nabla (\tilde u_{k+1} - \hat u_k )\|^2.
\end{align*}
This establishes \eqref{eqn:equiv} and finishes the proof.
\end{proof}
Next, we bound the difference of $u_k$ between successive iterations. 
\begin{lemma}%[$0<\sigma<1$]
\label{lemma20} 
Let Assumptions \ref{ass:aapn1}, \ref{ass:ukbd} and $\sigma C_A \max\{1,L\} <1$ hold,
then we have
\begin{multline}
\label{eqn:newterr}
\|\nabla (u_{k+1} - u_{k})\| \le 
 \frac{6\nu^{-1}M\sigma^2 C_A^{2} }{1-\sigma C_A \max\{1,L\}}  \|\nabla (u_k -u_{k-1})\|^2\\
+\frac{(C_k^{-2}+6)\nu^{-1}M}{1-\sigma C_A \max\{1,L\}} \|\nabla (\tilde u_{k+1} - \hat u_{k})\|^2 + \mathcal{O}( \|\nabla (\tilde u_{k+1}- \hat u_{k})\|^4),
\end{multline}
where $C_k>0$ is given in Lemma \ref{lemma10}.
\end{lemma}
\begin{proof}
Subtracting \eqref{wnaa} with $k$ from \eqref{wnaa} with $k+1$ yields
\begin{align}
\label{eqn:equi03}
\nu(\nabla (u_{k+1} - u_{k}), \nabla v)  + b^*(u_{k} - \hat u_{k}, \hat u_{k+1} - \hat u_k,v) + b^*( u_{k+1} - u_k,\hat u_{k+1}, v)  \nonumber\\
- b^*(\hat u_{k+1} - \hat u_k, \hat u_{k+1} -u_{k}, v) +b^*(\hat u_{k+1} , u_{k+1} - u_k,v) =0.
\end{align}
Setting $v = u_{k+1} - u_k$ eliminates the last term. Applying  \eqref{bstarbound}, \eqref{aaukbd} and followed by triangle inequality and inequality $2ab \le a^2 + b^2$,  for any $a,b\ge 0$ gives
\begin{align*}
 & \nu(1- \sigma C_A\max\{1,L\} ) \|\nabla (u_{k+1} - u_k)\|  \\
 \le &  M\|\nabla (\hat u_{k+1} - \hat u_k)\| \|\nabla (u_{k} - \hat u_{k})\| + M \|\nabla (\hat u_{k+1} - \hat u_{k})\|\|\nabla (\hat u_{k+1} - u_{k})\| \\
 \le &2 M\|\nabla (\hat u_{k+1} - \hat u_k)\| \|\nabla (u_{k} - \hat u_{k})\| + M \|\nabla (\hat u_{k+1} - \hat u_{k})\|^2 \\
 \le & M \|\nabla (u_{k} - \hat u_{k})\|^2 +  2M\|\nabla (\hat u_{k+1} - \hat u_k)\|^2.
 \end{align*}
 Then utilizing \eqref{eqn:haterr} and $(a+b+c)^2 \le 3(a^2 +b^2+c^2)$ for any $a,b,c\ge0$ and followed by Lemma \ref{lemma10}, we obtain
  \begin{align*}
 & \nu(1- \sigma C_A\max\{1,L\} ) \|\nabla (u_{k+1} - u_k)\|  \\
 \le & M \|\nabla (u_{k} - \hat u_{k})\|^2 +  6M  \left( \|\nabla (\tilde u_{k+1}- \hat u_{k})\|^2 + \sigma^2 C_A^2  \|\nabla (u_k - u_{k-1})\|^2\right) \\
 &  + 6M (1- \beta_{k+1})^2 \theta_{k+1}^2 \|\nabla (\tilde u_{k+1} -u_k)\|^2 \\
 \le & M (C_k^{-2}+6) \|\nabla (\tilde u_{k+1} - \hat u_{k})\|^2 
 + 6M\sigma^2 C_A^2 \|\nabla (u_k - u_{k-1})\|^2 
 \\ &
 + 6 \nu^{-2} M^3 (1- \beta_{k+1})^2 \theta_{k+1}^2 C_k^{-4} 
  \|\nabla (\tilde u_{k+1} -\hat u_k)\|^4.
\end{align*}
Dividing both sides by $ \nu(1- \sigma C_A\max\{1,L\} )$, we obtain \eqref{eqn:newterr} and complete the proof.
\end{proof}

Finally, we are ready to prove a quadratic convergence result.  The quadratic convergence is proven for differences of $(\tilde u_{k+1} - \hat u_k)$, as this is how the analysis naturally leads us, i.e., the Step 1 solution minus the Step 2 solution of the previous iteration.  This can still be considered as quadratic convergence of the algorithm in the usual sense, if we consider Step 1 on iteration 1 to be the initial guess instead, and reorder Steps 1,2,3 to be Steps 2,3,1.
\begin{thm}[$m=1$] 
\label{thm:aapn11}
Let $\alpha_k \neq 0$, Assumptions \ref{ass:aapn1}, \ref{ass:ukbd} and $\sigma C_A\max\{1,L\}, \sigma <1$ hold, we have
\begin{multline}
\label{eqn:conv1}
\|\nabla (\tilde u_{k+1} - \hat u_k)\| \le  
 \frac{ \nu^{-1}M C_{k-1}^{-2}C_k}{1-\sigma} \left( \theta_k\beta_k \sigma + (1-\beta_k ) C_A \right) \|\nabla ( \tilde u_{k} - \hat u_{k-1})\|^2 
 \\
+(1-\beta_k ) \frac{ C_A \nu^{-1} M C_{k-2}^{-2}}{1-\sigma} \|\nabla (\tilde u_{k-1} -\hat u_{k-2})\|^2
\\
+ \text{higher order terms of } \{ \tilde u_j - \hat u_{j-1} \}_{j=2}^k \text{ and } u_1 - u_0,
\end{multline}
where $\theta_k, C_A, L$ are defined in \eqref{eqn:theta1}, \eqref{eqn:aa1} and \eqref{ukbd2_beta} respectively, and $C_k$ are some constants independent of $\nu, h$.
\end{thm}
\begin{remark}
First of all, Algorithm \ref{alg:aapn1} converges quadratically for $0<\beta_{k+1} \le1$.
Evidently, $\beta_{k+1}=1$ optimizes the convergence as fewer residual terms are on the RHS of \eqref{eqn:conv1}, which reduces to  
\begin{multline}
\label{eqn:conv1beta}
\|\nabla (\tilde u_{k+1} - \hat u_k)\| \le  
 \frac{ \nu^{-1}M C_{k-1}^{-2}C_k}{1-\sigma}\sigma  \theta_k \|\nabla ( \tilde u_{k} - \hat u_{k-1})\|^2 
\\+ \text{higher order terms of } \{ \tilde u_j - \hat u_{j-1} \}_{j=2}^k \text{ and } u_1 - u_0.
\end{multline}
Besides that, the solutions $\{u_{k+1}\}$ are globally stable due to the discussion in Remark \ref{stab1}.  In addition, AAPicard-Newton method with $m=1$ is superior to Algorithm \ref{alg:pn} as $\theta_k <1$ when $\alpha_k\neq 0$ and $Re =\nu^{-1}$ large enough, leading to a small convergence rate. Otherwise, the higher-order terms may be dominant and decelerate the convergence. 
\end{remark}
\begin{proof}
From \eqref{hatu_dfn}, we rewrite  $ u_{k} - \hat u_k$ as 
\begin{align*}
u_k - \hat u_k  = &
u_k - \tilde u_{k} + \alpha_k(\tilde u_k - \tilde u_{k-1}) + (1-\beta_k) w_k^\alpha, 
\end{align*} 
and would like to construct an equation of $ u_{k} - \hat u_k$ below. We begin by subtracting \eqref{wnaa} with $u_{j-1}$ from \eqref{wpaa} with $\tilde u_j$ gives
\begin{align}
\nu(\nabla (\tilde u_j - u_{j-1}), \nabla v) + b^*(u_{j-1}, \tilde u_{j} -  u_{j-1},v)  + b^*(u_{j-1} - \hat u_{j-1}, u_{j-1} -\hat u_{j-1},  v) &=0.
\label{eqn:betak01}
\end{align}
Adding $(1-\alpha_k)\times$\eqref{eqn:betak01} with $j = k$ and  $\alpha_k\times$\eqref{eqn:betak01} with $j = k-1$ gives an equation of $w^\alpha_k$
\begin{multline}
\label{eqn:betak02}
\nu(\nabla w^\alpha_k, \nabla v) 
+  b^*(u_{k-1}, w_k^\alpha, v) 
- \alpha_kb^*(u_k - u_{k-1} , \tilde u_{k-1} - u_{k-2},v)  
 \\
+ (1-\alpha_k)  b^*(u_{k-1} - \hat u_{k-1}, u_{k-1} -\hat u_{k-1},  v)
+ \alpha_k  b^*(u_{k-2} - \hat u_{k-2}, u_{k-2} -\hat u_{k-2},  v)
=0.
\end{multline}
Adding $(1-\alpha_{k})\times$ \eqref{wpaa} with $\tilde u_{k}$ and $\alpha_{k}\times$ \eqref{wpaa} with $\tilde u_{k-1}$ yields
\begin{multline*}
\nu(\nabla ((1-\alpha_k) \tilde u_k + \alpha_k \tilde u_{k-1}) , \nabla v)
  - b^*(w_k^\alpha, \tilde u_k, v)
   +b^*((1-\alpha_k) \tilde u_k + \alpha_k \tilde u_{k-1}, \tilde u_k,v) \\
-\alpha_k b^*(u_{k-2},\tilde u_k - \tilde u_{k-1} ,v)
= \langle f,v\rangle. 
\end{multline*}
Subtracting it  from the sum of equation \eqref{wnaa} with $u_k$ and $(1-\beta_k)\times $ \eqref{eqn:betak02}, we obtain an equation of $u_k - \hat u_k$
\begin{multline*}
\nu(\nabla (u_k -\hat u_k),\nabla v)
=
- b^*(\hat u_k, u_k -\hat u_k, v)
 - b^*(u_k, \hat u_k,v)
-b^*(w_k^\alpha, \tilde u_k, v) 
 + b^*((1-\alpha_k) \tilde u_k + \alpha_k \tilde u_{k-1}, \tilde u_k,v)
 \\
  -\alpha_k b^*(u_{k-2},\tilde u_k - \tilde u_{k-1} ,v)
  -(1-\beta_k)  b^*(u_{k-1}, w_k^\alpha, v) 
+ (1-\beta_k) \alpha_kb^*(u_k - u_{k-1} , \tilde u_{k-1} - u_{k-2},v)  
 \\
- (1-\beta_k) (1-\alpha_k)  b^*(u_{k-1} - \hat u_{k-1}, u_{k-1} -\hat u_{k-1},  v)
- (1-\beta_k) \alpha_k  b^*(u_{k-2} - \hat u_{k-2}, u_{k-2} -\hat u_{k-2},  v).
\end{multline*}
Combining the third, fifth and sixth terms gives
\begin{align*}
& - b^*(u_k, \hat u_k,v)
 + b^*((1-\alpha_k) \tilde u_k + \alpha_k \tilde u_{k-1}, \tilde u_k,v)
 -\alpha_k b^*(u_{k-2},\tilde u_k - \tilde u_{k-1} ,v)
 \\
 = & - b^*(u_k - \hat u_k - (1-\beta_k)w_k^\alpha, \tilde u_k,v)
- b^*(u_k ,\hat u_k -  \tilde u_k ,v)
 -\alpha_k b^*(u_{k-2},\tilde u_k - \tilde u_{k-1} ,v)
\\
= & - b^*(u_k - \hat u_k, \tilde u_k,v)
 + (1-\beta_k) b^*(w_k^\alpha, \tilde u_k, v)
+ (1-\beta_k) b^*( u_k, w_k^\alpha,v)
+ \alpha_k b^*(u_k - u_{k-2}, \tilde u_k -\tilde u_{k-1},v),
\end{align*}
thanks to \eqref{hatu_dfn},
and then the equation of $u_k -\hat u_k$ becomes
\begin{multline*}
\nu(\nabla (u_k -\hat u_k),\nabla v)
=
- b^*(\hat u_k, u_k -\hat u_k, v) 
 -\beta_kb^*(w_k^\alpha, \tilde u_k, v) 
- b^*(u_k - \hat u_k, \tilde u_k,v)
\\
+\alpha_k b^*(u_k - u_{k-2}, \tilde u_k -\tilde u_{k-1},v)
  + (1-\beta_k)  b^*(u_k - u_{k-1}, w_k^\alpha, v) 
+ (1-\beta_k) \alpha_kb^*(u_k - u_{k-1} , \tilde u_{k-1} - u_{k-2},v)  
 \\
- (1-\beta_k) (1-\alpha_k)  b^*(u_{k-1} - \hat u_{k-1}, u_{k-1} -\hat u_{k-1},  v)
- (1-\beta_k) \alpha_k  b^*(u_{k-2} - \hat u_{k-2}, u_{k-2} -\hat u_{k-2},  v).
\end{multline*}
Setting $v = u_ k- \hat u_k$ eliminates the second term, and applying inequalities \eqref{bstarbound}, \eqref{eqn:aa1} and Lemma \ref{lemma:pic}  gives
\begin{align*}
& \nu (1-\sigma) \|\nabla (u_k - \hat u_k)\| 
\\
\le &
\beta_k M \|\nabla w_k^\alpha\|  \cdot \nu^{-1} \|f\|_{-1}  
+ |\alpha_k| M\|\nabla (u_k - u_{k-2})\| \cdot \sigma \|\nabla ( u_{k-1} -  u_{k-2})\|
\\
&
+ (1-\beta_k) M \|\nabla (u_k - u_{k-1})\| \|\nabla w_k^\alpha\| 
+ (1-\beta_k) C_A M \|\nabla (u_k - u_{k-1})\| \|\nabla (\tilde u_{k-1} - u_{k-2})\| 
\\
&
+ (1-\beta_k)C_AM\| \nabla (u_{k-1} -\hat u_{k-1})\|^2 
+ (1-\beta_k)C_A M \|\nabla (u_{k-2}- \hat u_{k-2})\|^2.
\end{align*}
Utilizing \eqref{eqn:theta1}, Lemma \ref{lemma10} and the inequality $2ab\le a^2 + b^2 $, we bound 
 the second and fourth terms by 
\begin{align*}
& \beta_k M \|\nabla w_k^\alpha\| \|\nabla \tilde u_k\| + (1-\beta_k) M \|\nabla (u_k - u_{k-1})\| \|\nabla w_k^\alpha\| 
\\
\le &
 \beta_k  \nu^{-1}\|f\|_{-1} M  \theta_k \|\nabla (\tilde u_k - u_{k-1})\|  + (1-\beta_k) M\|\nabla (u_k - u_{k-1})\| \cdot  \theta_k \|\nabla (\tilde u_k - u_{k-1})\| 
 \\
 \le & 
 \beta_k \sigma  \theta_k  MC_{k-1}^{-2}  \|\nabla ( \tilde u_{k} - \hat u_{k-1})\|^2
+ (1-\beta_k) \theta_k \nu^{-1}  M^2 C_{k-1}^{-2}\|\nabla (u_k - u_{k-1})\| \|\nabla ( \tilde u_{k} - \hat u_{k-1})\|^2
\\ 
\le & 
 \beta_k \sigma   \theta_k  MC_{k-1}^{-2}  \|\nabla ( \tilde u_{k} - \hat u_{k-1})\|^2
 + \frac12 (1-\beta_k) \theta_k \nu^{-1}  M^2\|\nabla (u_k - u_{k-1})\|^2 
 \\ &
 + \frac12 (1-\beta_k) \theta_k \nu^{-1}  M^2 C_{k-1}^{-4} \|\nabla ( \tilde u_{k} - \hat u_{k-1})\|^4,
\end{align*}
the third term by 
\begin{align*}
 & \sigma  C_A M\|\nabla (u_k - u_{k-2})\| \|\nabla ( u_{k-1} -  u_{k-2})\| 
 \\
\le & \sigma C_A M   \left( \|\nabla (u_{k} - u_{k-1})\|  \|\nabla (u_{k-1} - u_{k-2})\| +  \|\nabla (u_{k-1} - u_{k-2})\|^2 \right)
\\  \le & 
\frac12  C_A \sigma M  \left( \|\nabla (u_{k} - u_{k-1})\|^2 +3 \|\nabla (u_{k-1} - u_{k-2})\|^2 \right),
\end{align*}
 the fifth term by 
\begin{align*}
&  (1-\beta_k) C_A M \|\nabla (u_k - u_{k-1})\| \|\nabla (\tilde u_{k-1} - u_{k-2})\| 
\\
\le &
(1-\beta_k)C_A \nu^{-1}M^2\|\nabla (u_k - u_{k-1})\| \cdot C_{k-2}^{-2}\|\nabla (\tilde u_{k-1} - \hat u_{k-2})\|^2
\\ 
\le & 
\frac12 (1-\beta_k) C_A \nu^{-1}M^2\|\nabla (u_k - u_{k-1})\|^2
+\frac12 (1-\beta_k) C_A \nu^{-1}M^2 C_{k-2}^{-4} \|\nabla (\tilde u_{k-1} - \hat u_{k-2})\|^4,
\end{align*}
and the last two terms by 
\begin{align*}
& (1-\beta_k)C_A M\| \nabla (u_{k-1} -\hat u_{k-1})\|^2 
+ (1-\beta_k) C_A M \|\nabla (u_{k-2}- \hat u_{k-2})\|^2
\\
\le & 
(1-\beta_k ) C_A M C_{k-1}^{-2} \|\nabla (\tilde u_k -\hat u_{k-1})\|^2 
+ (1-\beta_k ) C_A M C_{k-2}^{-2} \|\nabla (\tilde u_{k-1} -\hat u_{k-2})\|^2.
\end{align*}
Combining the above five inequalities, we obtain
\begin{align*}
&  (1-\sigma)C_k^{-1} \|\nabla (\tilde u_{k+1} - \hat u_k)\| 
\\
\le & 
  \nu^{-1}M C_{k-1}^{-2} \left( \theta_k\beta_k \sigma + (1-\beta_k ) C_A \right) \|\nabla ( \tilde u_{k} - \hat u_{k-1})\|^2 
+  (1-\beta_k ) C_A \nu^{-1} M C_{k-2}^{-2} \|\nabla (\tilde u_{k-1} -\hat u_{k-2})\|^2
\\ &
  + \mathcal{O}( \|\nabla (u_k -u_{k-1})\|^2  +  \|\nabla (u_{k-1} - u_{k-2})\|^2)
  +(1-\beta_k)\cdot \mathcal{O} \left(  \|\nabla (\tilde u_k - \hat u_{k-1})\|^4 
  + \|\nabla (\tilde u_{k-1} - \hat u_{k-2})\|^4\right).
\end{align*}
Applying Lemma \ref{lemma20} recursively to the fourth term yields
\begin{align*}
&  \mathcal{O}( \|\nabla (u_k -u_{k-1})\|^2  +  \|\nabla (u_{k-1} - u_{k-2})\|^2)
\\ \le &
 \mathcal{O}( \|\nabla (u_{k-1} -u_{k-2})\|^4  +  \|\nabla (u_{k-2} - u_{k-3})\|^4)
 + \sum\limits_{j = k-1}^k\mathcal{O}( \|\nabla (\tilde u_j - \hat u_{j-1})\|^4 + \|\nabla (\tilde u_j - \hat u_{j-1})\|^8)
 \\ \le & \cdots 
 \\ \le &
  \mathcal{O}( \|\nabla (u_{2} -u_{1})\|^{2^{k-1}}  +  \|\nabla (u_{1} - u_{0})\|^{2^{k-1}}) 
+ \text{ higher order terms of }  \{\tilde u_j - \hat u_{j-1}\}_{j=2}^k 
\\ \le & 
 \mathcal{O}(\|\nabla (u_{1} - u_{0})\|^{2^{k}} + \|\nabla (u_{1} - u_{0})\|^{2^{k-1}}) 
+ \text{ higher order terms of }  \{\tilde u_j - \hat u_{j-1}\}_{j=2}^k .
\end{align*}
Combining the previous two inequalities yields \eqref{eqn:conv1} and we complete the proof.
\end{proof}

\subsection{AAPicard-Newton $m= 2$}
In this subsection, we analyze the AAPicard-Newton method with $m=2$. 
We add an additional previous residual term $\tilde u_{k-1} - u_{k-2}$ in the Anderson optimization step (Step 2 of Algorithm \ref{alg:aapn1}). Although this modification retains the quadratic convergence order, it improves the convergence rate of the iterative method and may enlarge the domain of convergence. Similarly, we can add more previous residuals in the Anderson step to  shrink the convergence rate further and possibly enlarge the domain of convergence. We will discuss the case $m>2$ in the following subsection.

The algorithm of AAPicard-Newton method with depth $m=2$ is stated below.

\begin{alg}[AAPicard-Newton $m=2$]
\label{alg:aapn2}
The AAPicard-Newton iteration with $m=2$ consists of applying the composition of the Newton and Anderson accelerated Picard iteration for solving  Navier-Stokes equations: $g_{N}\circ g_{AP}$, i.e.,
\begin{enumerate}[Step 1:]
\item[Step 1:] Find $\tilde u_{k+1} = g_P(u_{k})$ by finding $\tilde  u_{k+1}\in V$ 
satisfying \eqref{wpaa} for all $v\in V$.
\item[Step 2:] For a selected damping factor $ 0< \beta_{k+1} \le 1$,  set  
\begin{align}
\label{hatu_dfn2}
\hat u_{k+1} =& \beta_{k+1} \big( (1-\alpha_{k+1}^1 -\alpha_{k+1}^2)  \tilde u_{k+1} + \alpha_{k+1}^1 \tilde u_{k} + \alpha_{k+1}^2 \tilde u_{k-1} \big) \nonumber\\
&+ (1-\beta_{k+1}) \big((1-\alpha_{k+1}^1 -\alpha_{k+1}^2)   u_{k} + \alpha_{k+1}^1  u_{k-1} + \alpha_{k+1}^2  u_{k-2}  \big) 
\nonumber \\
= & 
(1-\alpha_{k+1}^1 -\alpha_{k+1}^2)  \tilde u_{k+1} + \alpha_{k+1}^1 \tilde u_{k} + \alpha_{k+1}^2 \tilde u_{k-1}
- (1-\beta_{k+1} ) w_{k+1,2}^\alpha ,
\end{align}
where $\alpha_{k+1}^1, \alpha_{k+1}^2$ minimizes 
\begin{align}
\label{eqn:min2}
\|\nabla w_{k+1,2}^\alpha \| \coloneqq
\left\|   (1-\alpha_{k+1}^1 -\alpha_{k+1}^2) \nabla (\tilde u_{k+1} - u_k) + \alpha_{k+1}^1 \nabla (\tilde u_k - u_{k-1} ) + \alpha_{k+1}^2 \nabla (\tilde u_{k-1} - u_{k-2}) \right\|.
\end{align}
\item[Step 3:] Find $u_{k+1} = g_N(\hat u_{k+1})$  by finding $u_{k+1}\in V$ satisfying \eqref{wnaa}  for all $v\in V$.
\end{enumerate}
\end{alg}
Similarly, let $\alpha_{k+1}^1, \alpha_{k+1}^2$ minimize \eqref{eqn:min2}, we define the Anderson gain with $m=2$ as
\begin{align}
\label{eqn:theta2}
\theta_{k+1,2} =& \frac{\|\nabla w_{k+1,2}^\alpha \| }{\|\nabla (\tilde u_{k+1} - u_k)\|}
.
\end{align}
Clearly, $0\le \theta_{k+1,2}  \le 1$ and $\theta_{k+1,2} =1$ if and only if $\alpha_{k+1}^1 =\alpha_{k+1}^2 =0$, which implies Algorithm \ref{alg:aapn2} is back to Algorithm \ref{alg:pn}. Moreover,   $\alpha_{k+1}^2 \neq 0$ evinces $\theta_{k+1,2} <\theta_{k+1} \le 1$, where $\theta_{k+1}$ is defined in \eqref{eqn:theta1}. Otherwise, Algorithm \ref{alg:aapn2} is back to either Algorithm \ref{alg:pn} or Algorithm \ref{alg:aapn1}. For the rest of this subsection, we always assume that $\alpha_{k+1}^2 \neq 0.$

Like Assumption \ref{ass:aapn1}, we make an uniform bounded assumption for $\{\alpha_{k+1}^i\}$ as below.
\begin{assumption}[$m=2$]
\label{ass:aapn2}
There exists a constant $C_A>0$ such that for all $k$
\begin{align}
|1-\alpha_{k+1}^1 - \alpha_{k+1}^2| + |\alpha_{k+1}^1| + |\alpha_{k+1}^2| & \le C_A,
\label{eqn:aa2} 
\end{align}
\end{assumption}

Compared to Algorithm \ref{alg:aapn1}, Step 1 and Step 3 from Algorithm \ref{alg:aapn2} are unchanged. We assume that Assumption \ref{ass:ukbd} holds in this subsection, and then the following results from the previous subsection are also satisfied here: Lemma \ref{lemma:pic}, equation \eqref{aaukbd}, Remark \ref{stab1}, and Lemma \ref{lemma10}.

  Nevertheless, we need to update equations \eqref{eqn:haterr}, Lemma \ref{lemma20} as below.
From \eqref{hatu_dfn2}, triangle inequality and followed by \eqref{eqn:picerr} and \eqref{eqn:theta2}, we have
 \begin{align}
 \label{eqn:haterrm2}
\|\nabla (\hat u_k - \hat u_{k-1})\| \le
& 
\|\nabla (\tilde u_k  - \hat u_{k-1})\| 
+\sigma C_A\|\nabla (u_{k-1} - u_{k-2})\| 
\nonumber\\
&
+ \sigma C_A\|\nabla (u_{k-2} - u_{k-3})\| 
+ (1-\beta_k) \theta_{k,2} \|\nabla (\tilde u_k -u_{k-1})\|.
\end{align}
\begin{lemma}[$m=2$]
\label{lemma20m2}
Let Assumption \ref{ass:aapn2}, \ref{ass:ukbd} and $\sigma C_A \max\{1,L\} <1$ hold,
then we have
\begin{align}
\|\nabla (u_{k+1} - u_k)\| \le &
\frac{8\nu^{-1}M\sigma^2 C_A^2}{1-\sigma C_A\max\{1,L\}} \left( \|\nabla (u_k - u_{k-1})\|^2 +  \|\nabla (u_{k-1} - u_{k-2})\|^2 \right)
\nonumber\\ &
+ 
\frac{ \nu^{-1}M (C_k^{-2} + 8)}{1-\sigma C_A\max\{1,L\} }  \|\nabla (\tilde u_{k+1} - \hat u_k)\|^2 
 + \mathcal{O}(\|\nabla (\tilde u_{k+1}- \hat u_k)\|^4).
\label{eqn:newterrm2}
\end{align} 
\end{lemma}
The proof is similar to Lemma \ref{lemma20} except replacing inequality \eqref{eqn:haterr} by \eqref{eqn:haterrm2}, and therefore omitted.

\begin{thm}[$m=2$]
\label{thm:aapn21}
Let $\alpha_k^2 \neq0$, Assumptions \ref{ass:aapn2} and \ref{ass:ukbd}, and $\sigma, \sigma C_A \max\{1,L\} <1$, we have 
\begin{multline}
\label{eqn:conv2}
\|\nabla (\tilde u_{k+1} - \hat u_k)\| \le 
\frac{\nu^{-1}MC_{k-1}^{-2}C_k}{1-\sigma} \left( \beta_k \sigma  \theta_{k,2} + (1-\beta_k) C_A \right) \|\nabla (\tilde u_k -\hat u_{k-1})\|^2
\\
+(1-\beta_k) \frac{C_A \nu^{-1}M}{1-\sigma} \sum\limits_{j =k-2}^{k-1} (C_{j-1}^{-2} \|\nabla (\tilde u_{j} -\hat u_{j-1})\|^2) 
\\
+ \text{higher order terms of } \{\tilde u_j - \hat u_{j-1}\}_{j=2}^{k} \text{ and } u_1 -u_0,
\end{multline}
where $\theta_{k,2}, C_A, L$ are defined in \eqref{eqn:theta2}, \eqref{eqn:aa2} and \eqref{ukbd2_beta} respectively, and $C_k$ are some constants independent of $\nu, h$.
\end{thm}
\begin{remark}
Similar to Theorem \ref{thm:aapn11}, Algorithm \ref{alg:aapn2} converges quadratically, and $\beta_{k+1}=1$ optimizes the convergence and provides global stability of the solutions when $\sigma, \sigma C_A<1$, and 
 \eqref{eqn:conv2} reduces to  
\begin{multline}
\label{eqn:conv2beta}
\|\nabla (\tilde u_{k+1} - \hat u_k)\| \le  
 \frac{ \nu^{-1}M C_{k-1}^{-2}C_k}{1-\sigma}\sigma  \theta_{k,2} \|\nabla ( \tilde u_{k} - \hat u_{k-1})\|^2 
\\+ \text{higher order terms of } \{ \tilde u_j - \hat u_{j-1} \}_{j=2}^k \text{ and } u_1 - u_0.
\end{multline}
From the inequality, we can tell that Algorithm \ref{alg:aapn2} is superior to Algorithm \ref{alg:pn} as $\theta_{k,2} <1$ when $\alpha_k^2\neq 0$ and $Re$ large enough. Otherwise, the higher-order terms may be dominant and decelerate the convergence.
Comparing \eqref{eqn:conv1beta} and \eqref{eqn:conv2beta},  we conclude that Algorithm \ref{alg:aapn2} is better than Algorithm \ref{alg:aapn1} for large Reynolds numbers as $\theta_{k,2} < \theta_k \le 1$. 
\end{remark}
\begin{proof}
From \eqref{hatu_dfn2}, 
we rewrite $ u_k - \hat u_k$ as 
\begin{align}
u_k -\hat u_k = &
u_k -(1-\alpha_k^1 -\alpha_k^2) \tilde u_k -\alpha_k^1 \tilde u_{k-1} -\alpha_k^2 \tilde u_{k-2} + (1-\beta_k)w_{k,2}^\alpha.
\label{eqn:m2eq1}
\end{align}
and would like to  construct an equation of $u_k - \hat u_k$ below. Adding  equation $(1-\alpha_k^1- \alpha_k^2) \times$ \eqref{eqn:betak01} with $j=k$, $\alpha_k^1\times $ \eqref{eqn:betak01} with $j=k-1$, and $\alpha_k^2\times$ \eqref{eqn:betak01} with $j=k-2$, we obtain an equation of $w_{k,2}^{\alpha}$
\begin{multline}
\nu (\nabla w_{k,2}^{\alpha}, \nabla v) 
+ b^*(u_{k-1}, w_{k,2}^{\alpha},v) 
- \alpha_k^1 b^*(u_{k-1}-u_{k-2} ,  \tilde u_{k-1}-u_{k-2},v)
 \\
- \alpha_k^2 b^*(u_{k-1}-u_{k-3}, \tilde u_{k-2} -u_{k-3},v) 
 + (1-\alpha_k^1 -\alpha_k^2) b^*(u_{k-1}-\hat u_{k-1}, u_{k-1}-\hat u_{k-1}, v) 
 \\
 + \alpha_k^1 b^*(u_{k-2} -\hat u_{k-2}, u_{k-2} - \hat u_{k-2}, v)
 +\alpha_k^2 b^*(u_{k-3}-\hat u_{k-3}, u_{k-3}-\hat u_{k-3},v)
 =0.
 \label{eqn:thm2eq1}
\end{multline}
Adding equation $(1-\alpha_k^1- \alpha_k^2) \times$ \eqref{wpaa} with $j=k$, $\alpha_k^1\times $ \eqref{wpaa} with $j=k-1$, and $\alpha_k^2\times$ \eqref{wpaa} with $j=k-2$ yields
\begin{multline*}
\nu (\nabla ((1-\alpha_k^1 -\alpha_k^2) \tilde u_k + \alpha_k^1 \tilde u_{k-1} +\alpha_k^2 \tilde u_{k-2} ), \nabla v) 
+  b^*((1-\alpha_k^1 -\alpha_k^2) \tilde u_k + \alpha_k^1 \tilde u_{k-1} +\alpha_k^2 \tilde u_{k-2} , \tilde u_k,v)
\\
-b^*(w_{k,2}^\alpha,\tilde u_k,v)
- \alpha_k^1 b^*(u_{k-2}, \tilde u_k - \tilde u_{k-1},v) 
- \alpha_k^2 b^*(u_{k-3}, \tilde u_k- \tilde u_{k-2}, v)
= \langle f,v\rangle.
\end{multline*}
Subtracting it from the sum of equation \eqref{wnaa} with $u_k$ and $(1-\beta_k)\times $\eqref{eqn:thm2eq1} produces an equation of $u_k -\hat u_k$
\begin{multline*}
\nu (\nabla (u_k -\hat u_k), \nabla v) 
= 
- b^*(\hat u_k, u_k -\hat u_k, v)
 - b^*(u_k, \hat u_k,v)
 +  b^*((1-\alpha_k^1 -\alpha_k^2) \tilde u_k + \alpha_k^1 \tilde u_{k-1} +\alpha_k^2 \tilde u_{k-2} , \tilde u_k,v)
 \\
 -b^*(w_{k,2}^\alpha,\tilde u_k,v)
- \alpha_k^1 b^*(u_{k-2}, \tilde u_k - \tilde u_{k-1},v) 
- \alpha_k^2 b^*(u_{k-3}, \tilde u_k- \tilde u_{k-2}, v)
- (1-\beta_k )b^*(u_{k-1}, w_{k,2}^{\alpha},v) 
\\
+ (1-\beta_k )\alpha_k^1 b^*(u_{k-1}-u_{k-2} ,  \tilde u_{k-1}-u_{k-2},v)
+(1-\beta_k )\alpha_k^2 b^*(u_{k-1}-u_{k-3}, \tilde u_{k-2} -u_{k-3},v) 
\\
 - (1-\beta_k )(1-\alpha_k^1 -\alpha_k^2) b^*(u_{k-1}-\hat u_{k-1}, u_{k-1}-\hat u_{k-1}, v) 
 -(1-\beta_k ) \alpha_k^1 b^*(u_{k-2} -\hat u_{k-2}, u_{k-2} - \hat u_{k-2}, v)
 \\
 -(1-\beta_k )\alpha_k^2 b^*(u_{k-3}-\hat u_{k-3}, u_{k-3}-\hat u_{k-3},v).
\end{multline*}
Combining the 3rd, 4th, 6th and 7th terms leads to 
\begin{align*}
 &- b^*(u_k, \hat u_k,v)
 +  b^*((1-\alpha_k^1 -\alpha_k^2) \tilde u_k + \alpha_k^1 \tilde u_{k-1} +\alpha_k^2 \tilde u_{k-2} , \tilde u_k,v)
 \\
&- \alpha_k^1 b^*(u_{k-2}, \tilde u_k - \tilde u_{k-1},v) 
- \alpha_k^2 b^*(u_{k-3}, \tilde u_k- \tilde u_{k-2}, v)
\\
= & 
 - b^*( u_k -\hat u_k - (1-\beta_k)w_{k,2}^\alpha, \tilde u_k, v) 
 - b^*(u_k, \hat u_k - \tilde u_k, v) 
  \\
&- \alpha_k^1 b^*(u_{k-2}, \tilde u_k - \tilde u_{k-1},v) 
- \alpha_k^2 b^*(u_{k-3}, \tilde u_k- \tilde u_{k-2}, v)
\\
=&
-b^*(u_k - \hat u_k, \tilde u_k,v) 
+ (1-\beta_k) b^*(w_{k,2}^\alpha, \tilde u_k, v)
+(1-\beta_k)b^*(u_k, w_{k,2}^\alpha,v)
\\&
+\alpha_k^1 b^*(u_k -u_{k-2}, \tilde u_k - \tilde u_{k-1},v)
+ \alpha_k^2 b^*(u_k - u_{k-3}, \tilde u_k - \tilde u_{k-2},v),
\end{align*}
thanks to \eqref{hatu_dfn2}, and then update the equation of $u_k - \hat u_k$ to
\begin{multline*}
\nu (\nabla (u_k -\hat u_k), \nabla v) 
= 
- b^*(\hat u_k, u_k -\hat u_k, v)
 -\beta_k b^*(w_{k,2}^\alpha,\tilde u_k,v)
-b^*(u_k - \hat u_k, \tilde u_k,v) 
\\
+(1-\beta_k)b^*(u_k- u_{k-1}, w_{k,2}^\alpha,v)
+\alpha_k^1 b^*(u_k -u_{k-2}, \tilde u_k - \tilde u_{k-1},v)
+ \alpha_k^2 b^*(u_k - u_{k-3}, \tilde u_k - \tilde u_{k-2},v)
\\
+(1-\beta_k ) \alpha_k^1 b^*(u_{k-1}-u_{k-2} ,  \tilde u_{k-1}-u_{k-2},v)
+(1-\beta_k )\alpha_k^2 b^*(u_{k-1}-u_{k-3}, \tilde u_{k-2} -u_{k-3},v) 
\\
 -(1-\beta_k ) (1-\alpha_k^1 -\alpha_k^2) b^*(u_{k-1}-\hat u_{k-1}, u_{k-1}-\hat u_{k-1}, v) 
 -(1-\beta_k ) \alpha_k^1 b^*(u_{k-2} -\hat u_{k-2}, u_{k-2} - \hat u_{k-2}, v)
 \\
 -(1-\beta_k )\alpha_k^2 b^*(u_{k-3}-\hat u_{k-3}, u_{k-3}-\hat u_{k-3},v).
\end{multline*}
Setting $v = u_k -\hat u_k$ eliminate the second term. Applying \eqref{bstarbound}, \eqref{eqn:aa2}, Lemma \ref{lemma:pic} and 
\eqref{eqn:theta2} gives
\begin{align*}
& \nu(1-\sigma ) \|\nabla (u_k - \hat u_k ) \|
\\
\le & 
\beta_k \nu \sigma \theta_{k,2}\|\nabla (\tilde u_k -u_{k-1})\|
+(1-\beta_k)M \|\nabla (u_k - u_{k-1}) \|  \cdot \theta_{k,2} \|\nabla (\tilde u_k -u_{k-1})\|
\\
&
+C_A M\|\nabla (u_k - u_{k-2})\| \cdot \sigma \|\nabla (u_{k-1} -u_{k-2} )\|
+ C_A M \|\nabla (u_k -u_{k-3})\| \cdot \sigma \|\nabla ( u_{k-1} - u_{k-3})\|
\\& 
+(1-\beta_k) C_AM \|\nabla (u_{k-1} - u_{k-2}) \|\|\nabla (\tilde u_{k-1} -u_{k-2})\| 
\\&
+ (1-\beta_k) C_A M \|\nabla (u_{k-1} -u_{k-3})\| \|\nabla (\tilde u_{k-2}- u_{k-3})\| 
+(1-\beta_k) C_A M \sum\limits_{j =k-3}^{k-1}\|\nabla (u_{j} -\hat u_{j})\|^2 .
\end{align*}
Utilizing triangle inequality, $(a+b)^2 \le 2 (a^2+ b^2)$ and $2ab\le a^2 + b^2$, we bound the RHS term-wisely as below
\begin{align*}
 & C_A M\|\nabla (u_k - u_{k-2})\| \cdot \sigma \|\nabla (u_{k-1} -u_{k-2} )\|
 \\ \le & 
 \sigma C_A M \left(  \|\nabla (u_k - u_{k-1})\|\|\nabla (u_{k-1} - u_{k-2})\| + \|\nabla (u_{k-1} - u_{k-2})\|^2\right)
 \\ \le & 
  \frac32 \sigma C_A M \|\nabla (u_{k-1} -u_{k-2})\|^2  
+ \frac12 \sigma C_AM \|\nabla (u_k - u_{k-1})\|^2,
\end{align*}
and
\begin{align*}
& C_A M \|\nabla (u_k -u_{k-3})\| \cdot \sigma \|\nabla ( u_{k-1} - u_{k-3})\|
\\ \le & 
\sigma C_A M \left(  \|\nabla ( u_{k} - u_{k-1})\| \cdot  \|\nabla ( u_{k-1} - u_{k-3})\| +  \|\nabla ( u_{k-1} - u_{k-3})\|^2 \right) 
\\ \le & 
  \frac32 \sigma C_A M ( \|\nabla (u_{k-1} -u_{k-2})\| + \|\nabla (u_{k-2}- u_{k-3})\| )^2  
+ \frac12 \sigma C_AM \|\nabla (u_k - u_{k-1})\|^2
\\ \le & 
 3 \sigma C_A M \|\nabla (u_{k-2} -u_{k-3})\|^2  
+ 3 \sigma C_A M \|\nabla (u_{k-1} -u_{k-2})\|^2
+ \frac12 \sigma C_AM \|\nabla (u_k - u_{k-1})\|^2.
\end{align*}
Now using Lemma \ref{lemma10}, we have the following bounds
\begin{align*}
 & \beta_k \nu \sigma \theta_{k,2}\|\nabla (\tilde u_k -u_{k-1})\|
+(1-\beta_k)M \|\nabla (u_k - u_{k-1}) \|  \cdot \theta_{k,2} \|\nabla (\tilde u_k -u_{k-1})\|
\\ \le & 
\beta_k \sigma MC_{k-1}^{-2} \theta_{k,2} \|\nabla (\tilde u_k -\hat u_{k-1})\|^2
+ \frac12 (1-\beta_k) \nu M^2 C_{k-1}^{-2} \theta_{k,2}  \|\nabla (\tilde u_k -\hat u_{k-1})\|^4
%\\ & 
+\mathcal{O}(\|\nabla (u_k -u_{k-1})\|^2), 
\\
 & (1-\beta_k) C_AM \|\nabla (u_{k-1} - u_{k-2}) \|\|\nabla (\tilde u_{k-1} -u_{k-2})\| 
 \\ \le & 
 (1-\beta_k) C_A \nu^{-1}M^2 C_{k-2}^{-2} \|\nabla (u_{k-1} - u_{k-2}) \|  \|\nabla (\tilde u_{k-1} -\hat u_{k-2})\|^2
 \\ \le &  
 \mathcal{O}(\|\nabla (u_{k-1} - u_{k-2})\|^2 + \|\nabla (\tilde u_{k-1} -\hat u_{k-2})\|^4),
\\
 &  (1-\beta_k) C_A M \|\nabla (u_{k-1} -u_{k-3})\| \|\nabla (\tilde u_{k-2}- u_{k-3})\| 
 \\ \le & 
  (1-\beta_k) C_A M  \|\nabla (u_{k-1} -u_{k-3})\|  \cdot \nu^{-1}M C_{k-3}^{-2} \|\nabla (\tilde u_{k-2}- \hat u_{k-3})\|^2
  \\ \le & 
(1-\beta_k) C_A \nu^{-1}M^2 C_{k-3}^{-2}  ( \|\nabla (u_{k-1} -u_{k-2})\| +  \|\nabla (u_{k-2} -u_{k-3})\|)^2
+ \mathcal{O}(\|\nabla (\tilde u_{k-2}- \hat u_{k-3})\|^4)
\\ \le & 
\mathcal{O}( \|\nabla (u_{k-1} -u_{k-2})\|^2 +  \|\nabla (u_{k-2} -u_{k-3})\|^2)
+ \mathcal{O}( \|\nabla (\tilde u_{k-2}- \hat u_{k-3})\|^4), 
 \end{align*}
 and 
 \begin{align*}
  (1-\beta_k) C_A M \sum\limits_{j =k-3}^{k-1}\|\nabla (u_{j} -\hat u_{j})\|^2 
\le 
(1-\beta_k) C_A M \sum\limits_{j =k-3}^{k-1} C_j^{-2}\|\nabla (\tilde u_{j+1} -\hat u_{j})\|^2. 
 \end{align*}
Combining the above seven inequalities, we obtain
\begin{multline*}
 \nu(1-\sigma ) C_k^{-1} \|\nabla (\tilde u_{k+1} - \hat u_k ) \|
\le 
MC_{k-1}^{-2} \left( \beta_k \sigma  \theta_{k,2} + (1-\beta_k) C_A \right) \|\nabla (\tilde u_k -\hat u_{k-1})\|^2
\\
+(1-\beta_k) C_AM \sum\limits_{j =k-2}^{k-1} (C_{j-1}^{-2} \|\nabla (\tilde u_{j} -\hat u_{j-1})\|^2) 
+  \sum\limits_{j=k-2}^k\mathcal{O}(  \|\nabla (u_{j} -u_{j-1})\|^2 )
+\sum\limits_{j = k-2}^k \mathcal{O}( \|\nabla (\tilde u_{j}- \hat u_{j-1})\|^4).
\end{multline*}
Applying Lemma \ref{lemma20m2} recursively to the fourth term produces
\begin{align*}
 & \sum\limits_{j=k-2}^k\mathcal{O}(  \|\nabla (u_{j} -u_{j-1})\|^2 )
 \\ \le  & 
\sum\limits_{j=k-3}^{k-1} \mathcal{O}(  \|\nabla (u_{j} -u_{j-1})\|^4 ) 
+\sum\limits_{j=k-2}^{k} \mathcal{O} (\|\nabla (\tilde u_j -\hat u_{j-1})\|^4 + \|\nabla (\tilde u_j -\hat u_{j-1})\|^8 )
\\ \le &
\sum\limits_{j=k-4}^{k-2} \mathcal{O}(  \|\nabla (u_{j} -u_{j-1})\|^{2^3} ) 
+ \sum\limits_{j=k-3}^{k-1} \mathcal{O} (\|\nabla (\tilde u_j -\hat u_{j-1})\|^{2^3} + \|\nabla (\tilde u_j -\hat u_{j-1})\|^{2^4}  )
\\&
+\sum\limits_{j=k-2}^{k} \mathcal{O} (\|\nabla (\tilde u_j -\hat u_{j-1})\|^4 + \|\nabla (\tilde u_j -\hat u_{j-1})\|^8 )
\\ 
\le &  \cdots 
\\ \le &
\sum\limits_{j=1}^{3} \mathcal{O}(  \|\nabla (u_{j} -u_{j-1})\|^{2^{k-2}} ) 
+ \text{ higher order terms of }\{\tilde u_j - \hat u_{j-1}\}_{j=2}^{k}
\\
\le & 
\sum\limits_{j=1}^{3} \mathcal{O}(  \|\nabla (u_{1} -u_{0})\|^{2^{k-3+j}} ) 
+ \text{ higher order terms of }\{\tilde u_j - \hat u_{j-1}\}_{j=2}^{k}.
\end{align*}
Combining the previous two inequalities and then dividing both sides by $\nu(1-\sigma) C_k^{-1}$ yields \eqref{eqn:conv2}. We complete the proof.
\end{proof}

\subsection{AAPicard-Newton $m >2$}
 By adding more previous residual terms in the optimization step (Step 2 in Algorithm \ref{alg:aapn1}, \ref{alg:aapn2}), we obtain the general AAPicard-Newton algorithm. Similar results can be obtained with additional analyses; thus, the proof is omitted here.
Now we present the algorithm and general convergence result of the AAPicard-Newton method below.
\begin{alg}[$m = 1, 2,3, \dots$]
\label{alg:aapnm}
The AAPicard-Newton iteration with $m=1,2,\dots$ consists of applying the composition of the Newton and Anderson accelerated Picard iteration for solving  Navier-Stokes equations: $g_{N}\circ g_{AP}$, i.e.,
\begin{enumerate}[Step 1:]
\item[Step 1:] Find $\tilde u_{k+1} = g_P(u_{k})$ by finding $\tilde  u_{k+1}\in V$ 
satisfying \eqref{wpaa} for all $v\in V$.
\item[Step 2:] For $0<\beta_{k+1} \le 1,$ let $m_k = \min\{k-1,m\}$, and  
\begin{multline*}
\hat u_{k+1} = \beta_{k+1} \left( \left(1- \sum\limits_{j =1}^{m_k}\alpha_{k+1}^j \right)  \tilde u_{k+1} + \sum\limits_{j =1}^{m_k }\alpha_{k+1}^j \tilde u_{k+1-j} \right) \\
+ (1-\beta_{k+1}) \left(  \left(1- \sum\limits_{j =1}^{m_k}\alpha_{k+1}^j \right)  u_{k} + \sum\limits_{j =1}^{m_k }\alpha_{k+1}^j  u_{k-j}\right)
\\
= \left(1- \sum\limits_{j =1}^{m_k}\alpha_{k+1}^j \right)  \tilde u_{k+1} + \sum\limits_{j =1}^{m_k }\alpha_{k+1}^j \tilde u_{k+1-j}  - (1-\beta_{k+1}) w_{k+1,m}^\alpha,
\end{multline*}
where $\alpha_{k+1}^j$ minimizes 
\begin{align*}
\|\nabla w_{k+1,m}^\alpha \| \coloneqq
\left\|   \left(1-\sum\limits_{j=1}^{m_k}\alpha_{k+1}^j \right) \nabla (\tilde u_{k+1} - u_k) + \sum\limits_{j=1}^{m_k } \alpha_{k+1}^j \nabla (\tilde u_{k+1-j} - u_{k-j} )  \right\|.
\end{align*}
\item[Step 3:] Find $u_{k+1} = g_N(\hat u_{k+1})$  by finding $u_{k+1}\in V$ satisfying \eqref{wnaa}  for all $v\in V$.
\end{enumerate}
\end{alg}
Clearly, Algorithm \ref{alg:aapnm} reduces to a lower depth Anderson acceleration or Algorithm \ref{alg:pn} if $\alpha_{k+1}^{m_k} =0$. Let $\alpha_{k+1}^j$ minimize $\|\nabla w_{k+1,m}^\alpha\| $ and define 
\begin{align}
\label{eqn:thetam}
\theta_{k+1,m} = \|\nabla w_{k+1,m}^\alpha \| / \|\nabla (\tilde u_{k+1} - \hat u_k)\| .
\end{align}
Thus $\theta_{k+1,m} <1$ if $\alpha_{k+1}^{m_k} \neq 0.$ Furthermore, 
we assume $\alpha_{k+1}^j$ satisfying the inequality 
\begin{align}
\label{eqn:aam}
  \left| 1-\sum\limits_{j=1}^{m_k}\alpha_{k+1}^j \right| + \sum\limits_{j=1}^{m_k } | \alpha_{k+1}^j | \le C_A,
\end{align} for some constant $C_A>0.$
With similar analyses as above, we have the quadratic convergence of AAPicard-Newton method.
\begin{thm}[$m = 1,2,3,\dots,$]
\label{thm:aapnm}
Let $\alpha_k^{m}\neq 0$, \eqref{eqn:aam}, Assumption \ref{ass:ukbd} and $\sigma C_A \max\{1,L\}, \sigma <1$ hold,
then we have
\begin{multline}
\label{eqn:convm}
\|\nabla (\tilde u_{k+1} - \hat u_k)\| \le 
\frac{\nu^{-1}MC_{k-1}^{-2}C_k}{1-\sigma} \left( \beta_k \sigma  \theta_{k,m} + (1-\beta_k) C_A \right) \|\nabla (\tilde u_k -\hat u_{k-1})\|^2
\\
+(1-\beta_k) \frac{C_A \nu^{-1}M}{1-\sigma} \sum\limits_{j =k-m}^{k-1} (C_{j-1}^{-2} \|\nabla (\tilde u_{j} -\hat u_{j-1})\|^2) 
\\
+ \text{higher order terms of } \{\tilde u_j - \hat u_{j-1}\}_{j=2}^{k} \text{ and } u_1 -u_0,
\end{multline}
where $\theta_{k,m}, C_A, L$ are defined in \eqref{eqn:thetam}, \eqref{eqn:aam} and \eqref{ukbd2_beta} respectively, and $C_k$ are some constants independent of $\nu, h$.
\end{thm}

\begin{remark}
\label{remark}
Likewise before, AAPicard-Newton method converges quadratically and $\beta_{k+1}=1$ optimizes the convergence and provides global stability when $\sigma, \sigma C_A<1$.
 Reducing \eqref{eqn:convm} to  
\begin{multline}
\label{eqn:convmbeta}
\|\nabla (\tilde u_{k+1} - \hat u_k)\| \le  
 \frac{ \nu^{-1}M C_{k-1}^{-2}C_k}{1-\sigma}\sigma  \theta_{k,m} \|\nabla ( \tilde u_{k} - \hat u_{k-1})\|^2 
\\+ \text{higher order terms of } \{ \tilde u_j - \hat u_{j-1} \}_{j=2}^k \text{ and } u_1 - u_0, 
\end{multline}
we found that Algorithm \ref{alg:aapnm} is superior to the Picard-Newton method and AAPicard-Newton methods with smaller depth $l (<m)$ when $\alpha_{k+1}^m\neq0$ and $Re=\nu^{-1}$ large enough, because of $\theta_{k,m}<\theta_{k,l}\le 1$ or $\theta_{k,m}<1$. 
Otherwise, the higher-order terms may be dominant and decelerate the convergence. 

 Therefore, we recommend $\beta_{k+1}=1$ for all $k$ of Algorithm \ref{alg:aapnm} in the numeral tests with large Reynolds numbers.
\end{remark}

\section{Numerical tests for AAPicard-Newton}
In this section, we perform three numerical tests -- a 2D lid-driven cavity, a 3D lid-driven cavity, and channel flow past a cylinder -- to see how the Anderson-accelerated Picard method would affect the convergence of Newton's iteration.

 \subsection{2D lid-driven cavity}
 For this test, we consider the AAPicard-Newton method applied to the 2D lid-driven cavity benchmark problem. The domain is the unit square $\Omega = (0, 1)^2$, and we employ $(P_2, P_1^{disc})$ Scott-Vogelius elements on a barycenter refined uniform triangular meshes with a mesh size of $h = 1/64$, resulting in 98.8K velocity degrees of freedom (dof) and 73.7K pressure dof. We set  $f = 0$ and implement Dirichlet boundary conditions
that enforce no-slip velocity on the sides and bottom, while imposing a velocity of  $\langle 1, 0\rangle^T$ on the top (lid).  The initial guess is defined as $u_0=0$ within the interior of the domain, while satisfying the boundary condition.  

A plot of the velocity solutions (computed using the AAPicard-Newton $m=1$ and $\beta_{k+1}=1$)  for varying $Re \coloneqq \nu^{-1} = 5000,10000,15000$, and $20000$ is presented in Figure \ref{fig1}.  We observe that an additional small eddy appears and grows larger in the bottom right corner as $Re$ exceeds $10000$. Moreover, an additional small eddy emerges in the bottom left corner when $Re = 20000$.  These observations are consistent with findings in the literature \cite{ECC05}.

 \begin{figure}[h!]
\includegraphics[width = .24\textwidth, height=.24\textwidth,viewport=65 45 465 420, clip]{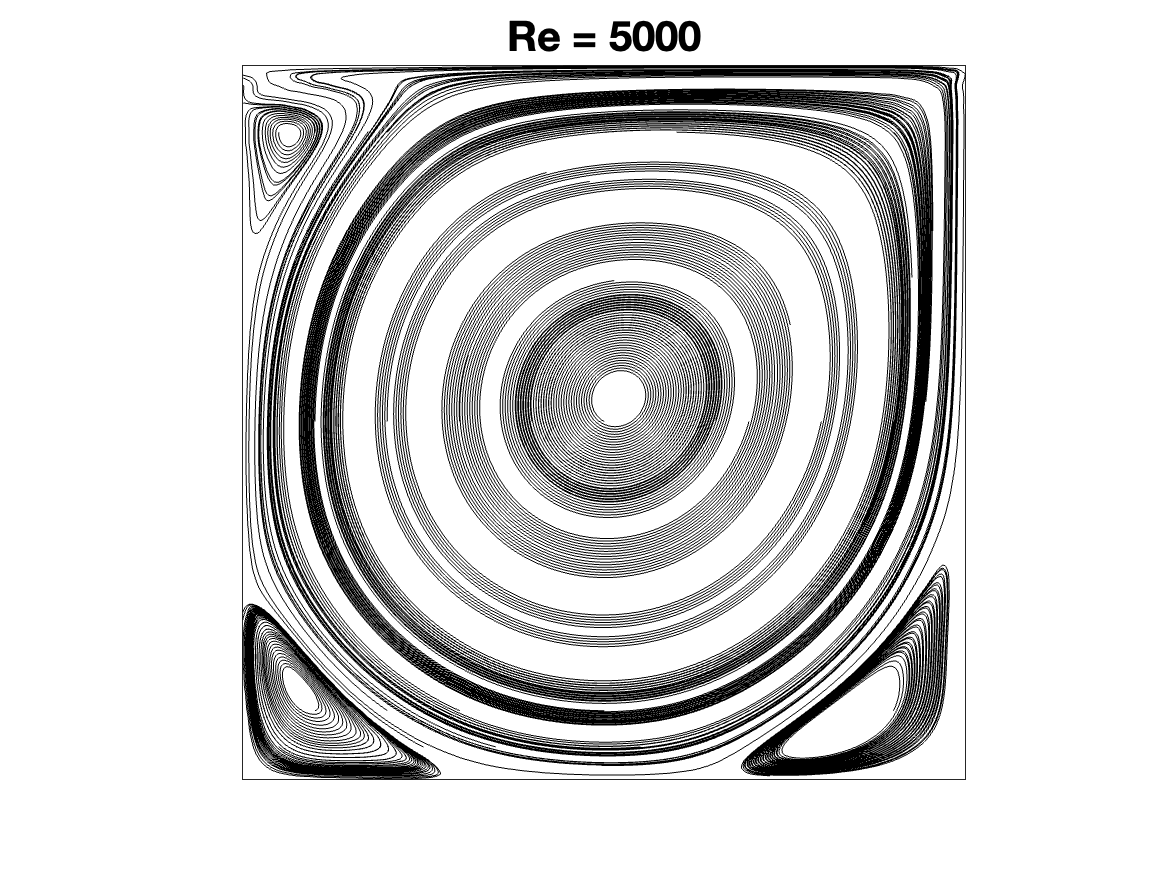}
 \includegraphics[width = .24\textwidth, height=.24\textwidth,viewport=65 45 465 420, clip]{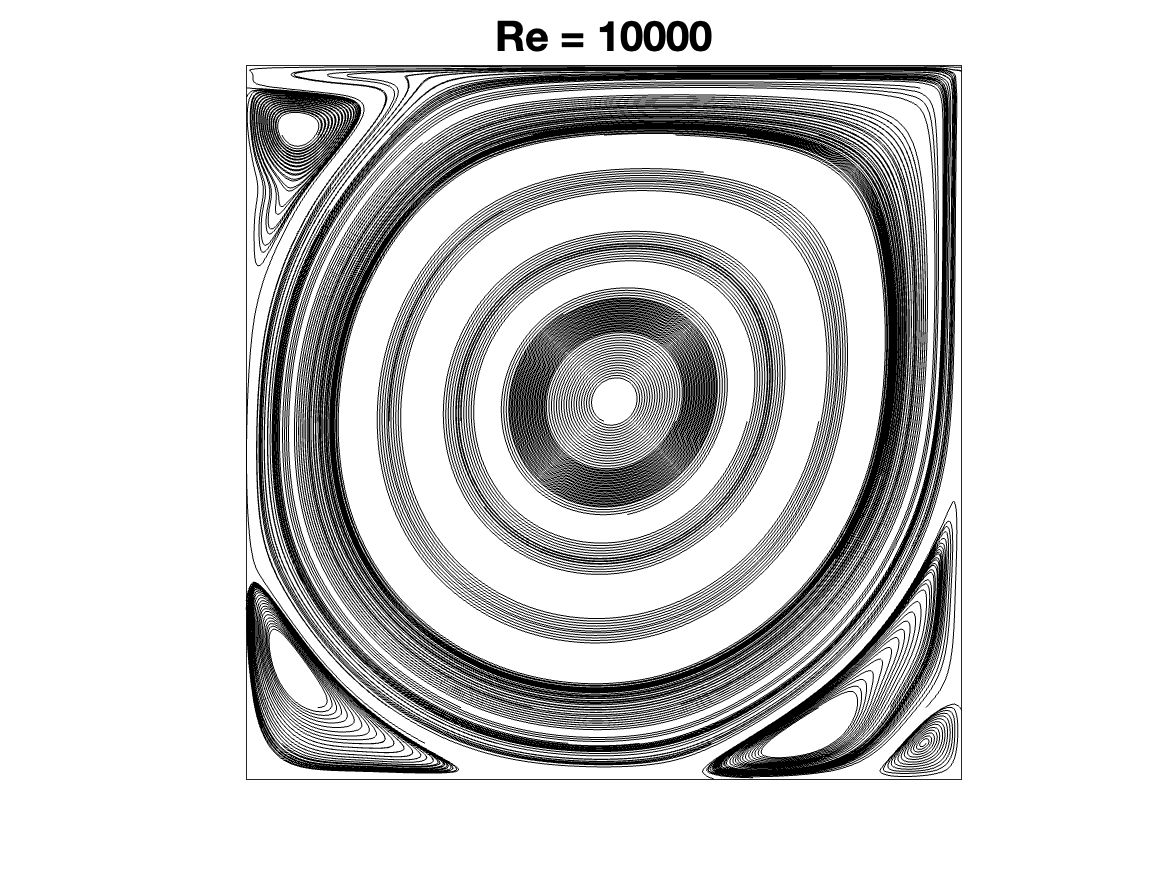}
\includegraphics[width = .24\textwidth, height=.24\textwidth,viewport=65 45 465 420, clip]{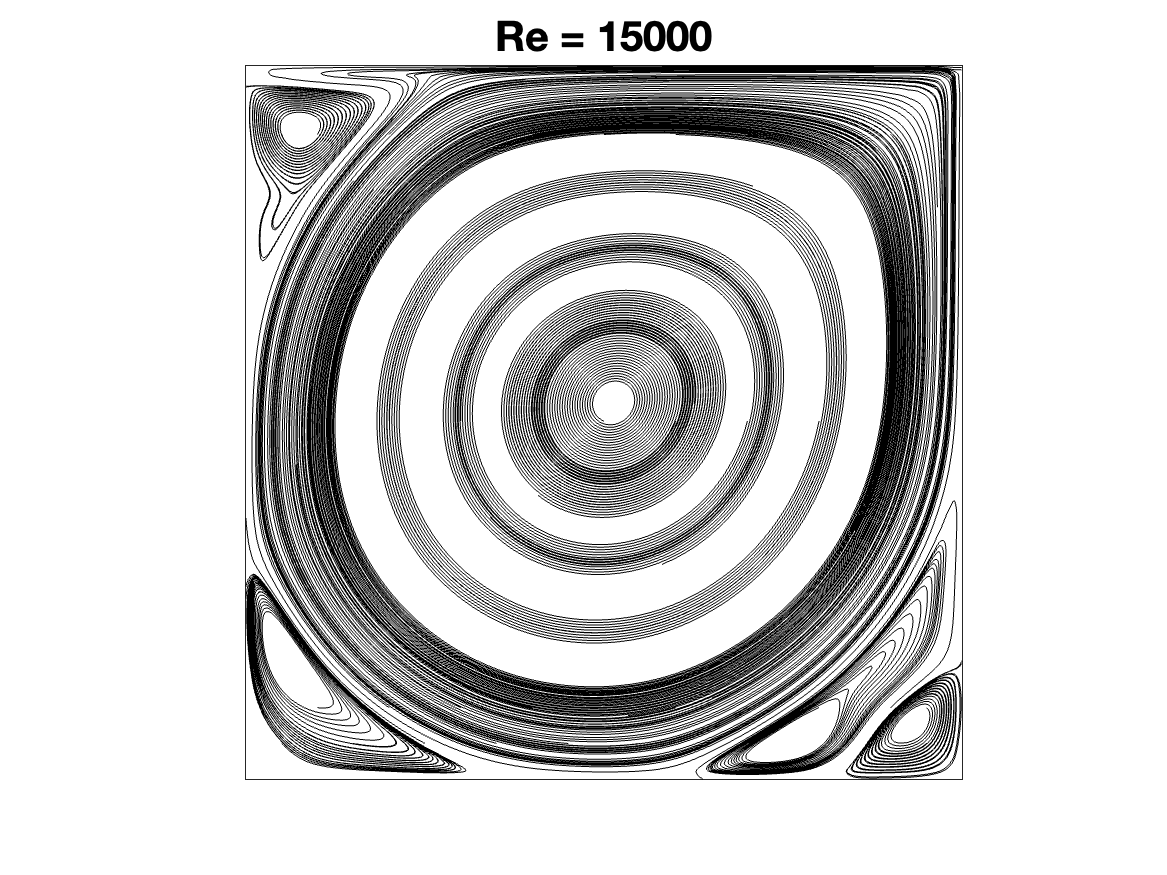}
\includegraphics[width = .24\textwidth, height=.24\textwidth,viewport=65 45 465 420, clip]{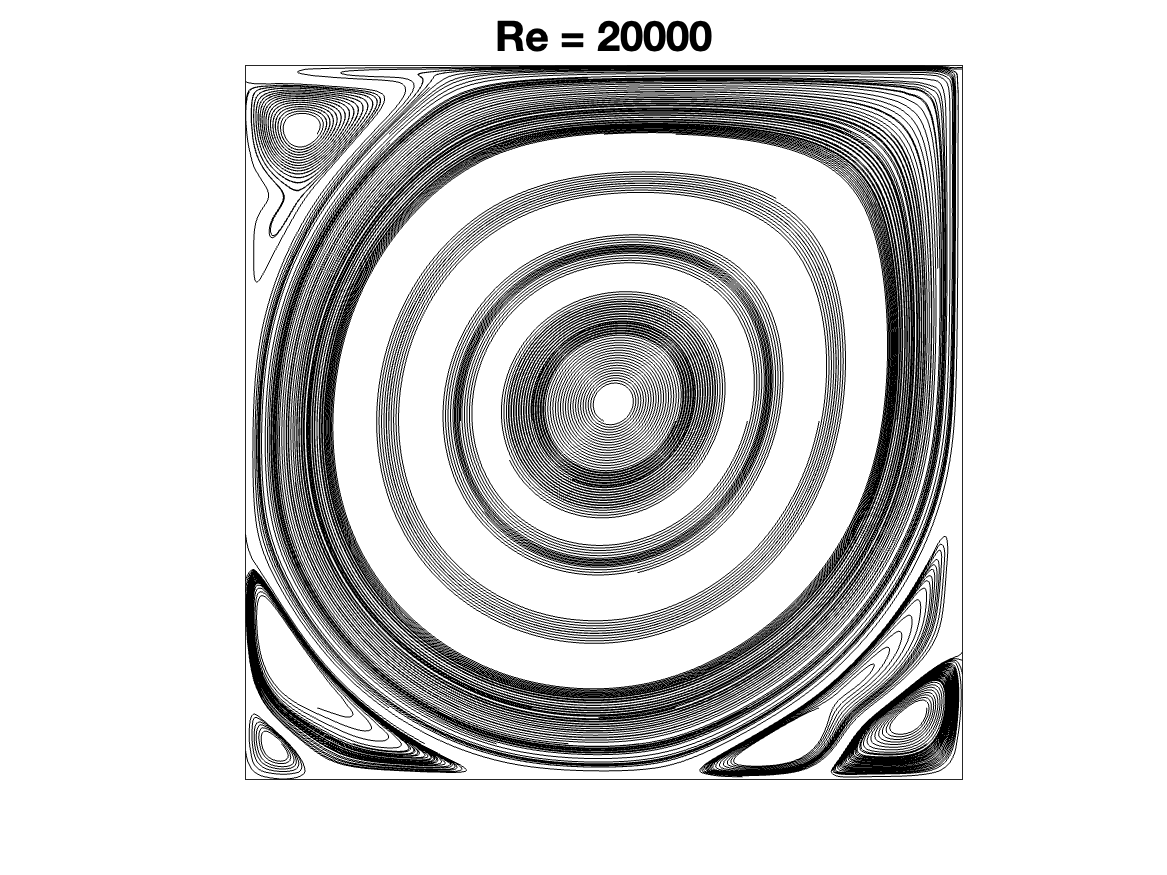}
\caption{Shown above are streamlines of velocity solutions found by the AAPicard-Newton method for the 2D lid-driven cavity problems with varying $Re$.    \label{fig1}}
 \end{figure}

The convergence plot of residuals $\|\nabla (\tilde u_{k+1} -\hat u_k)\|$ solved by the AAPicard-Newton method is shown in Figure \ref{fig2}. In this plot, $\beta_{k+1} \equiv 1$ is fixed, and $m=0$ refers to the usual Picard-Newton method Algorithm \ref{alg:pn}. We find that the AAPicard-Newton method converges quadratically; slows the convergence for medium $Re =5000$; and significantly enhances convergence for large Reynolds numbers, like $Re =10000,15000$, and $20000$.

The median values of $\theta_{m,k}$  are summarized in Table \ref{table:theta}, revealing that a larger Anderson depth $m$ results in a smaller Anderson gain $\theta$. Consequently,  for sufficiently high Reynolds number $Re$, the quadratic term dominates in \eqref{eqn:convmbeta}, and a larger depth $m$ accelerates convergence. 
Table \ref{table:comptime} presents the computational time solved by the Anderson step (Step 2), Picard step (Step 1), and Newton step (Step 3) in Algorithm \ref{alg:aapn1}, \ref{alg:aapn2}, and \ref{alg:aapnm}, under fixed $\beta =1$ and various Reynolds number $Re$ and Anderson depth $m$. We observe that the Anderson step solve time increases for large depth $m$, but it is still much cheaper than both the Picard solve and the Newton solve. Thus, we recommend using a large depth $m$ for high Reynolds number problems as the AAPicard-Newton method can dramatically cut the required number of iterations.

We also provide convergence plots for a fixed $Re = 10000$ with varied $\beta_{k+1}$ to examine how the relaxation parameter $\beta_{k+1}$ affects the performance of the AAPicard-Newton method, as shown in Figure \ref{fig:beta}. It is evident that $\beta_{k+1} \equiv 1$ optimizes convergence performance. All observations are in agreement with Theorem \ref{thm:aapnm} and  Remark \ref{remark}.

\begin{table}[h!]
\centering
\begin{tabular}{c|| ccccc }
&\multicolumn{5}{c}{median value of $\theta_{k,m}$}   \\
 $Re$  & $m=1$ & $m=2$ & $m=5$ & $m=10$ & $m=20$    \\  \hline
5000 &        0.9638 & 0.97668  & 0.5266  &0.2469 &0.1581       \\
10000 &  0.7823 & 0.7972 & 0.4961 & 0.1545 & 0.2492(F)       \\
15000 &   0.7663 &  0.7252 & 0.4514 & 0.3145 & 0.1359           \\
20000 &    0.6800(F) & 0.4447(F) & 0.2306& 0.2191 & 0.3856        
\end{tabular}
\caption{\label{table:theta}
Shown above are the median value of $\theta_{k,m}$ from the AAPicard-Newton method with fixed $\beta_{k+1} =1$ in Algorithm \ref{alg:aapn1}, \ref{alg:aapn2} and \ref{alg:aapnm}, various $m$ and $Re$. `F' means no convergence after 100 iterations.}
\end{table}

\begin{table}[h!]
\centering
\begin{tabular}{c|| ccccc || c|c}
& \multicolumn{7}{c}{Linear Solve Time (in sec)}\\ \hline
& \multicolumn{5}{c}{ Step 2}  &Step 1 & Step 3 \\
$Re$ &$m=1$ & $m=2$ & $m=5$ & $m=10$ & $m=20$ &  Picard solve & Newton solve \\ \hline\hline
$5000$ & 0.010& 0.010 & 0.013 &0.018& 0.034 & 2.2201 & 1.7974 \\
$10000$ &0.011  &0.011 & 0.018 & 0.020 & 0.055(F) & 2.0886 & 1.9405\\
$15000$ & 0.009 & 0.011 & 0.017 & 0.024 & 0.036 & 2.3679 & 1.7817\\
$20000$ & 0.024 (F) & 0.023(F) & 0.023  & 0.022 & 0.037 & 2.0158 & 1.7629\\
\end{tabular}
\caption{\label{table:comptime}
Shown above are the linear solve times in seconds (Step 2 Anderson optimization solve, Step 1 Picard solve, and Step 3 Newton solve) in Algorithm \ref{alg:aapn1}, \ref{alg:aapn2} and \ref{alg:aapnm}, with various $m$ and $Re$. `F' means no convergence after 100 iterations.}
\end{table}

 \begin{figure}[h!]
 \centering
 \includegraphics[width = .44\textwidth, height=.44\textwidth,viewport=2 2 550 420, clip]{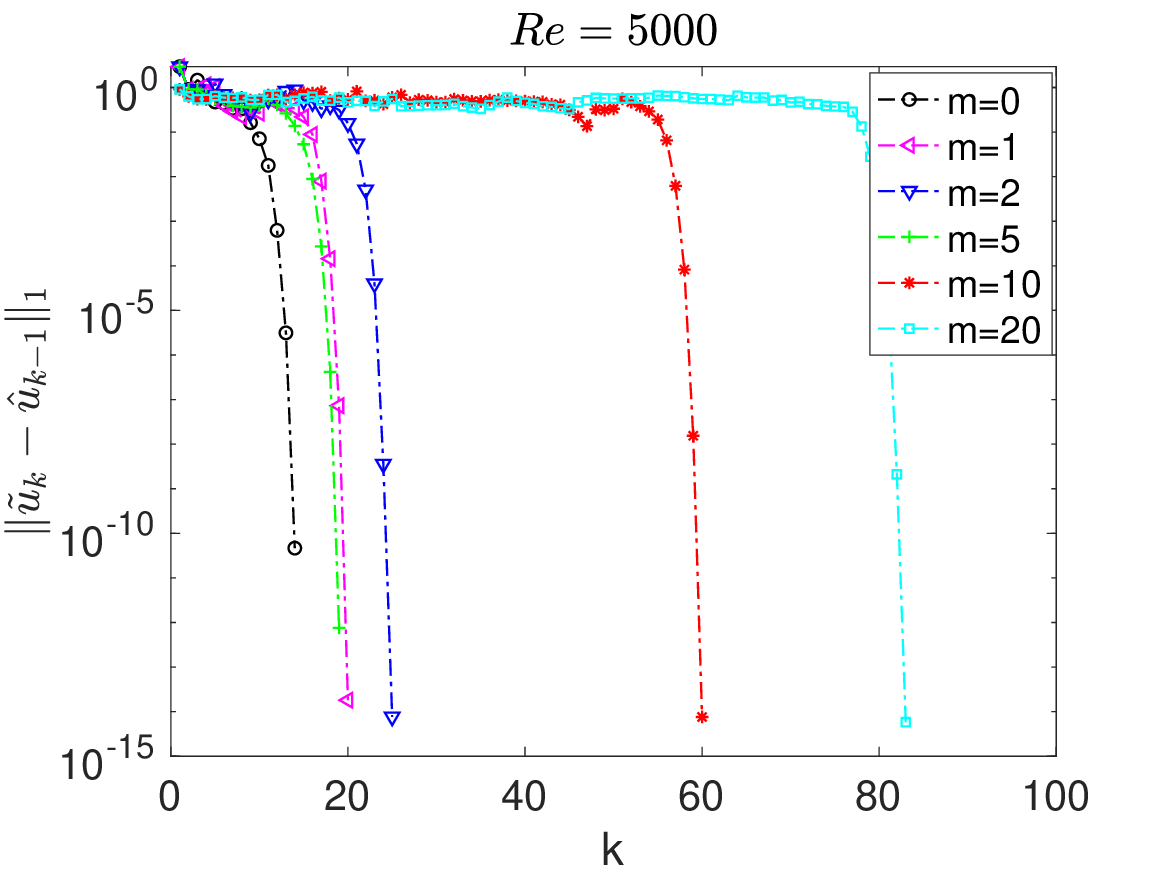}
 \includegraphics[width = .44\textwidth, height=.44\textwidth,viewport=2 2 550 420, clip]{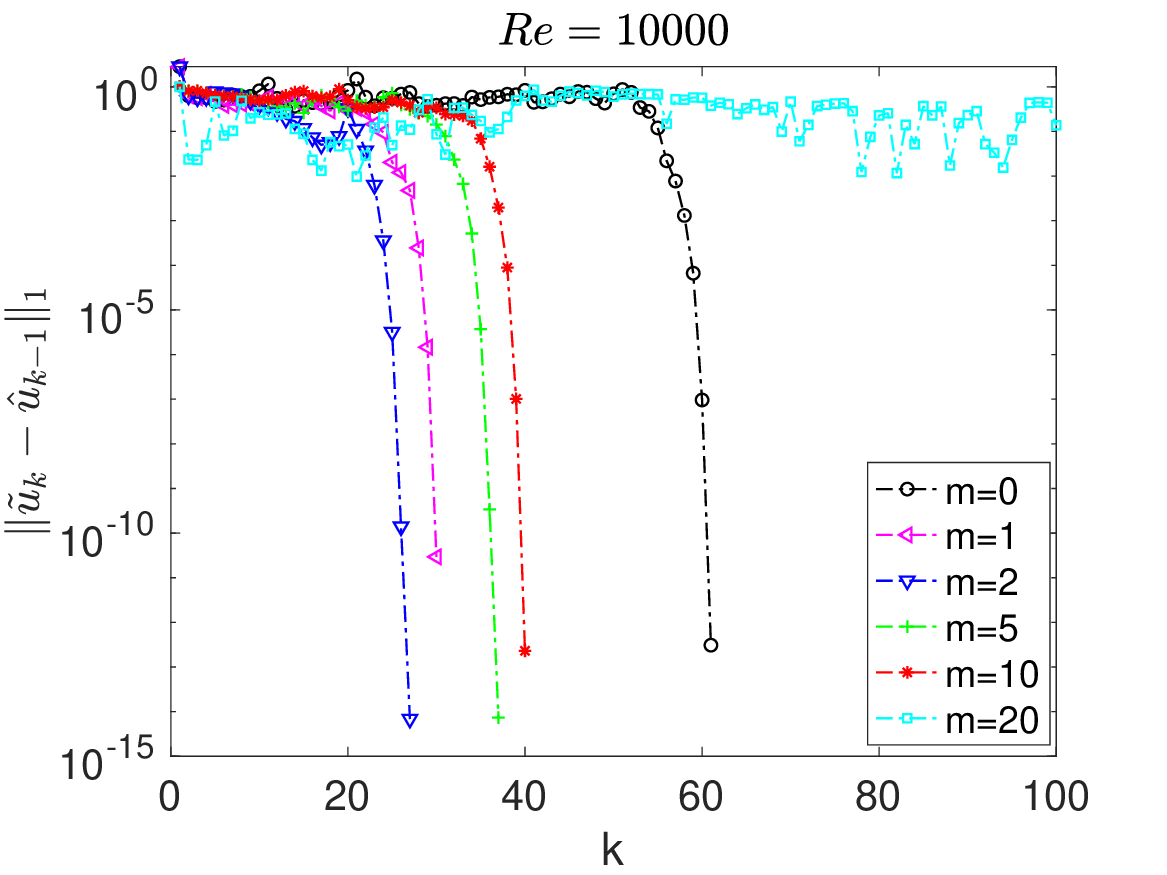}\\
\includegraphics[width = .44\textwidth, height=.44\textwidth,viewport=2 2 550 420, clip]{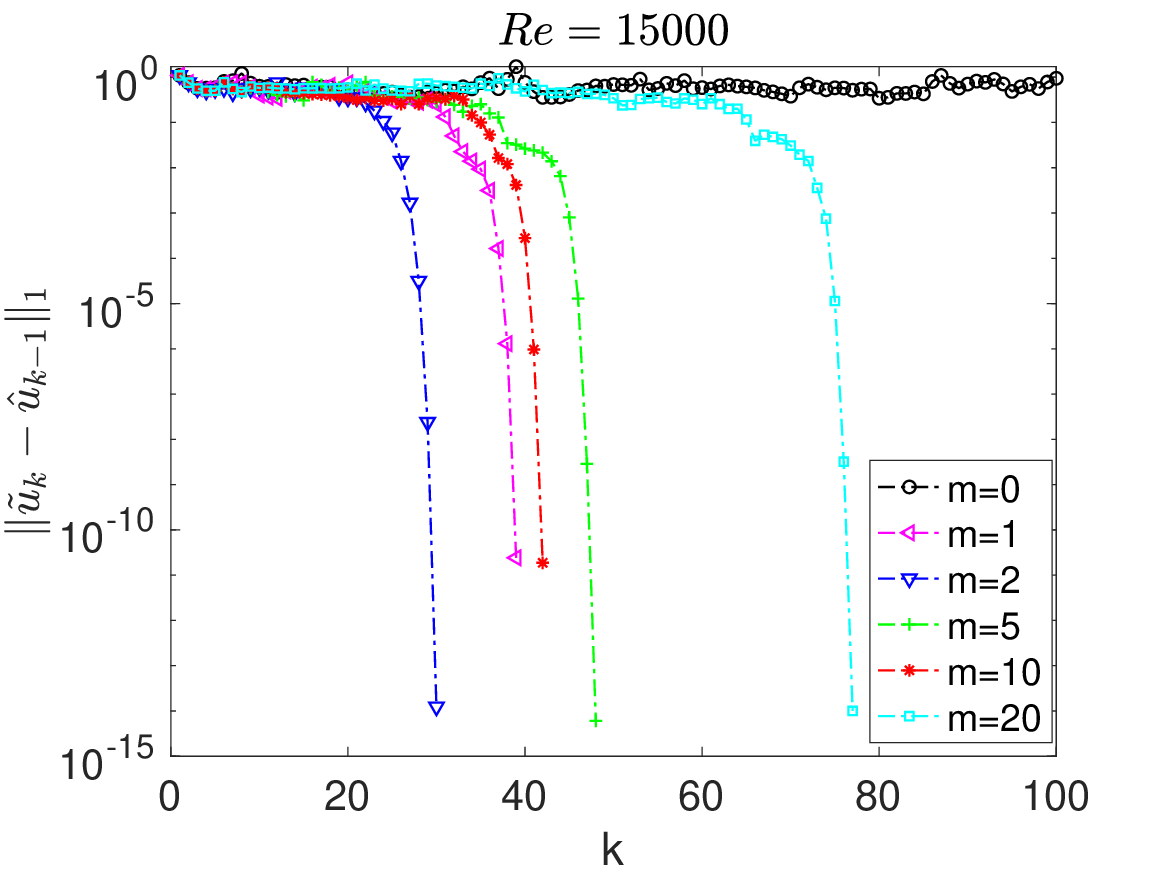}
\includegraphics[width = .44\textwidth, height=.44\textwidth,viewport=2 2 550 420, clip]{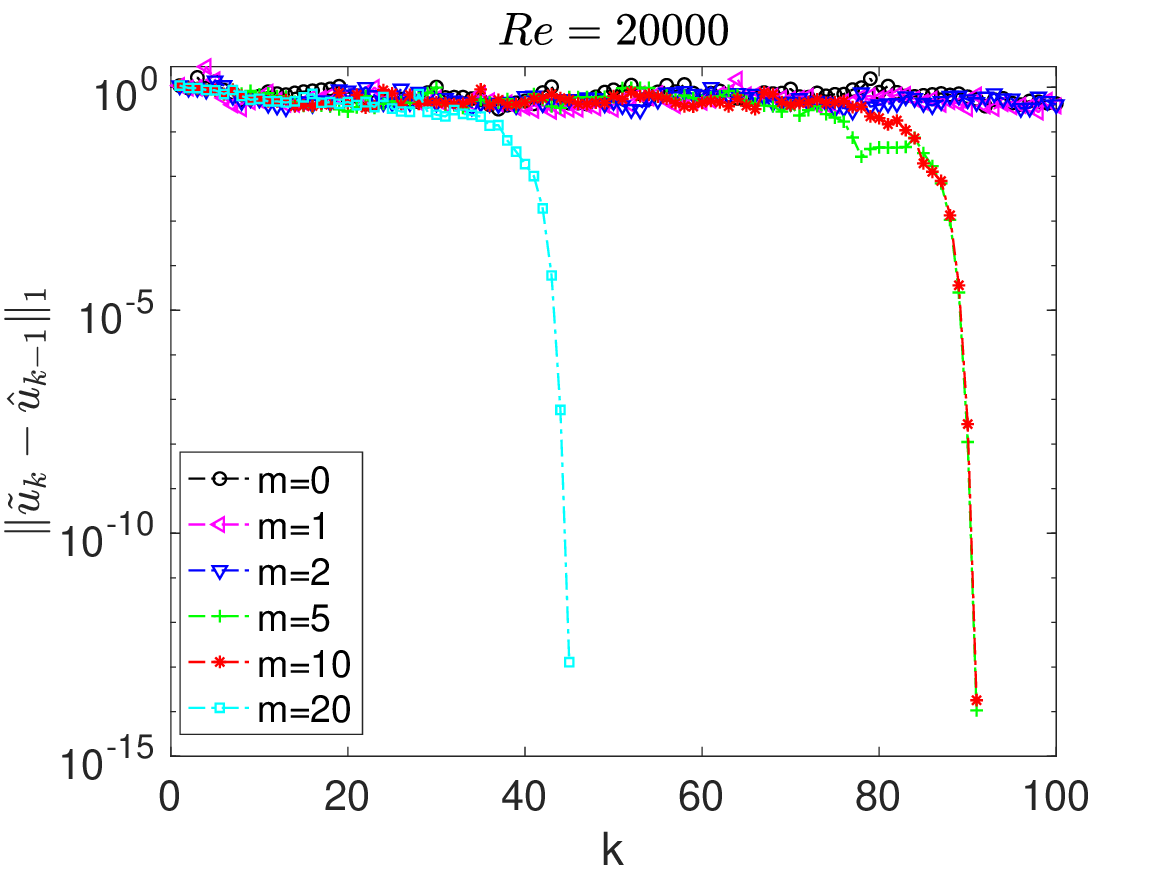}
\caption{\label{fig2} 
Shown above are the convergence plots of residual $\|\nabla (\tilde u_{k+1} -\hat u_k)\|$ found by AAPicard-Newton method with $\beta_{k+1}=1$, various Anderson acceleration depth $m$ and $Re$.}
 \end{figure}

\begin{figure}[h!]
\centering
\includegraphics[width = .44\textwidth, height=.44\textwidth,viewport=2 2 550 420, clip]{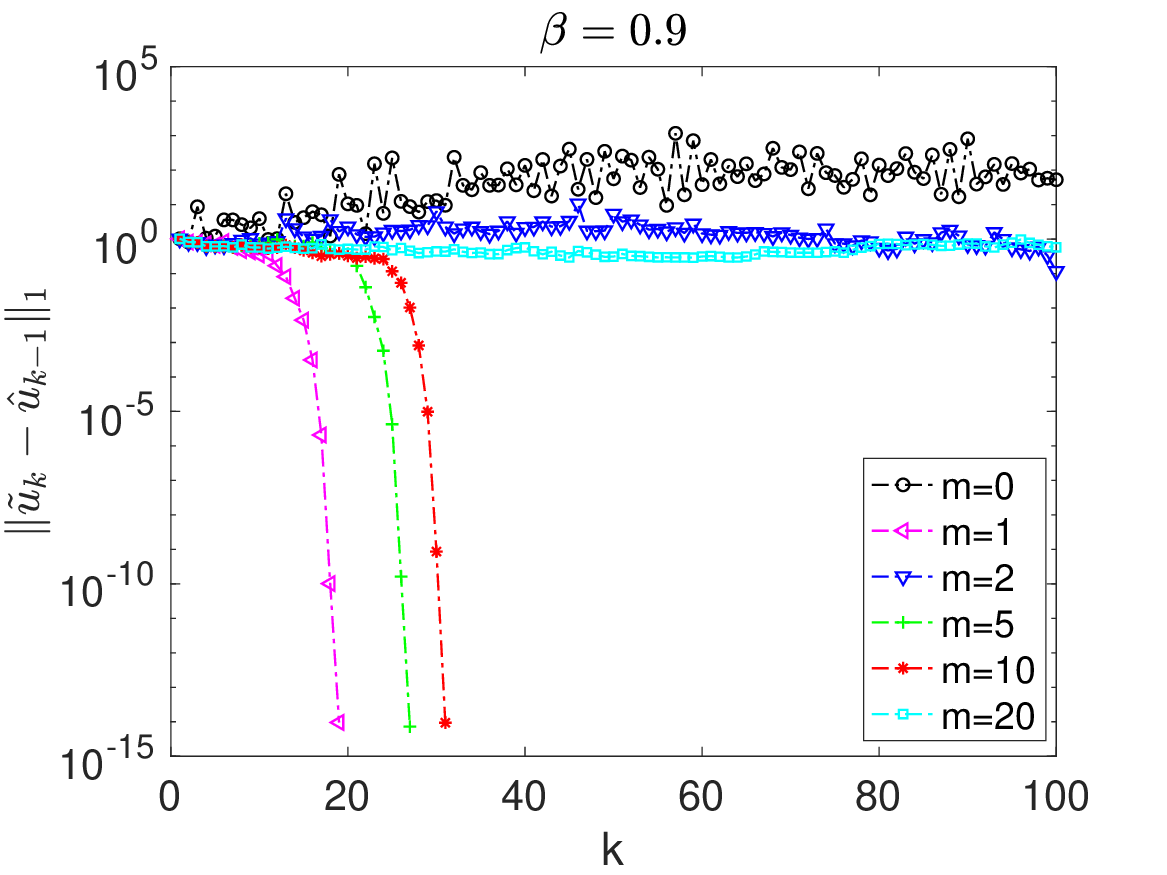}
 \includegraphics[width = .44\textwidth, height=.44\textwidth,viewport=2 2 550 420, clip]{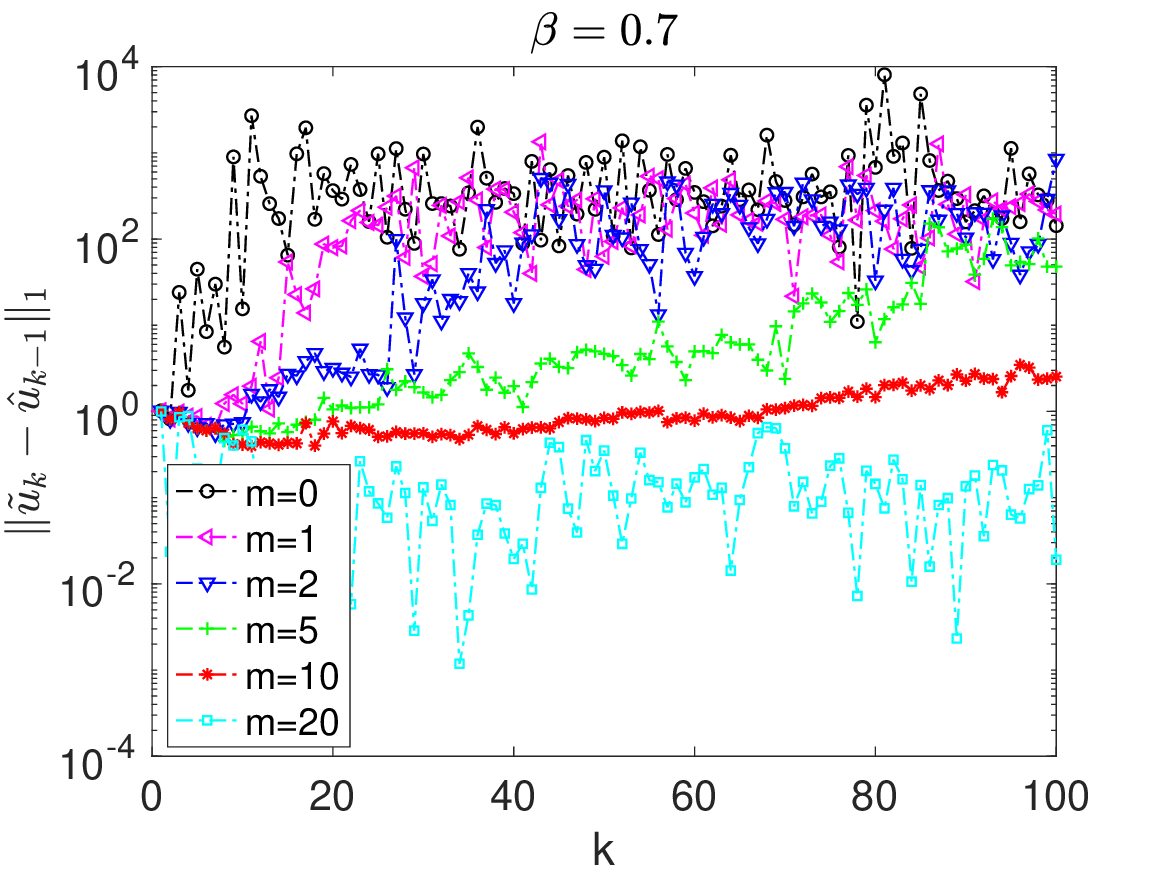}\\
\includegraphics[width = .44\textwidth, height=.44\textwidth,viewport=2 2 550 420, clip]{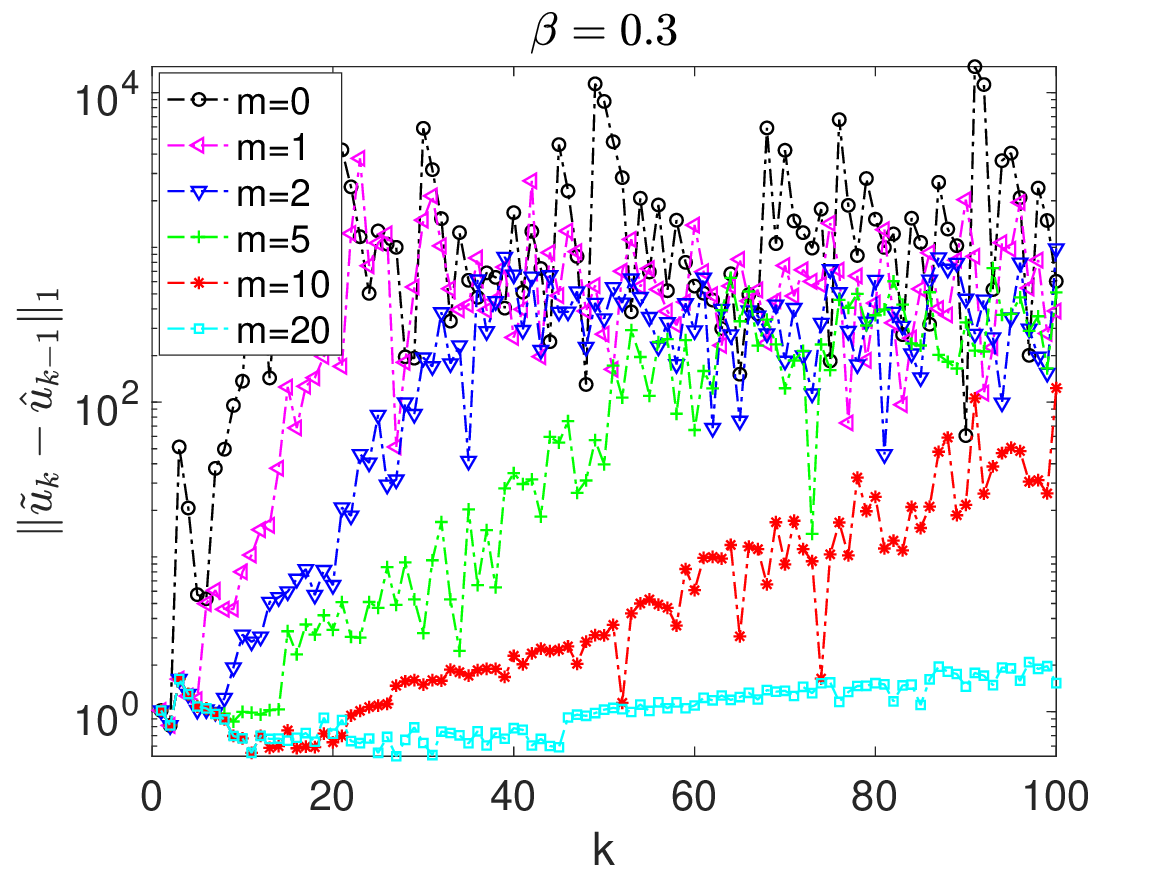}
\includegraphics[width = .44\textwidth, height=.44\textwidth,viewport=2 2 550 420, clip]{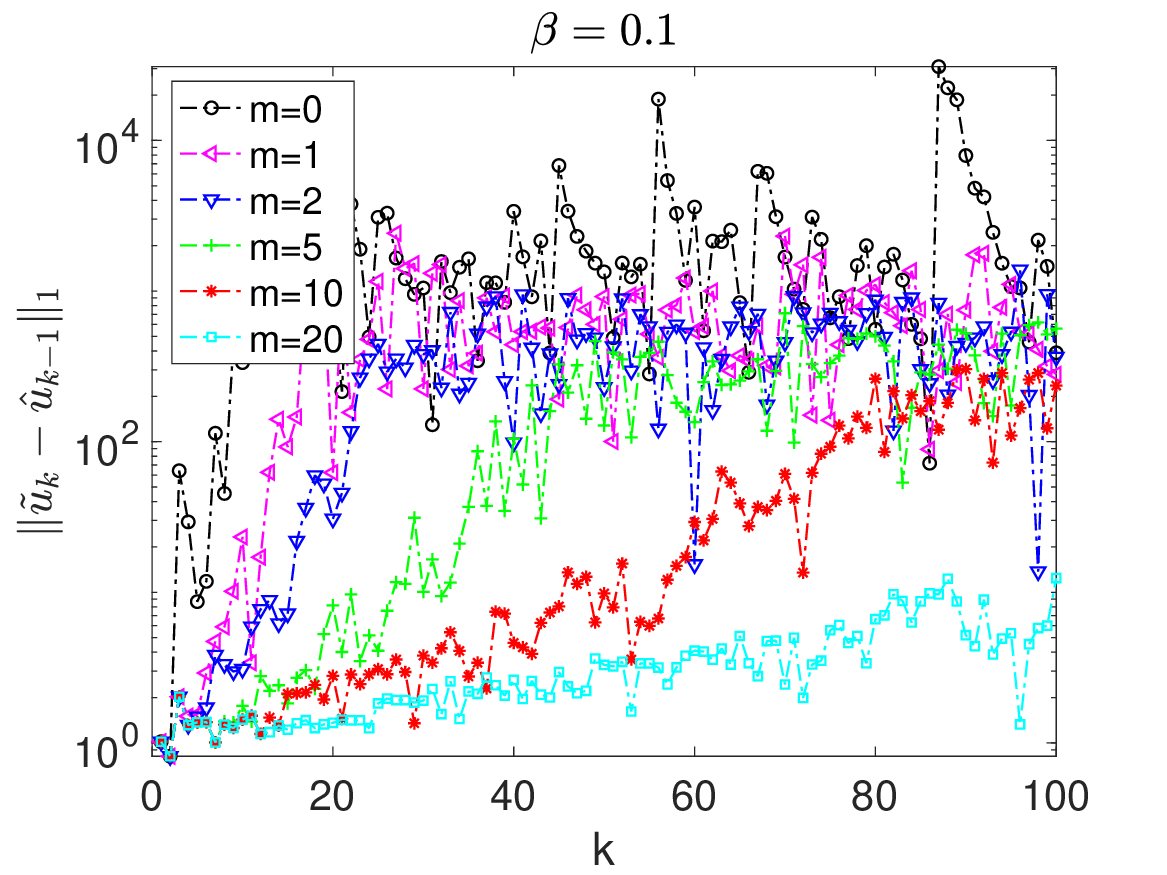}
\caption{
\label{fig:beta}
Shown above are the convergence plots of residual $\|\nabla (\tilde u_{k+1} -\hat u_k)\|$ found by AAPicard-Newton method with $Re =10000$, various Anderson acceleration depth $m$, and relaxation parameter $\beta_{k+1}$.}
\end{figure}

\subsection{3D lid-driven cavity}
Next, we examine the AAPicard-Newton method as applied to the 3D lid-driven cavity benchmark test. The domain is the unit cube $\Omega = (0,1)^3$.  We utilize $(P_3,P_2^{disc})$ Scott-Vogelius elements on a barycenter refined uniform tetrahedral mesh, which comprises a total of 796,722 dof. We set the forcing term $f=0$ and apply Dirichlet boundary conditions that impose the no-slip velocity on the sides and bottom, while applying a velocity of  $\langle 1,0,0 \rangle^T$ on the top lid. The initial guess is set as $u_0=0$ in the interior of the domain, satisfying the boundary conditions. At each iteration,  the solver employed is the incremental Picard-Yosida method as described in \cite{RVX19}. The velocity solutions obtained using the AAPicard-Newton method, with a depth $m=1$, are presented for varying $Re = \nu^{-1} = 100, 400$, and $1000$ and are illustrated in Figure \ref{fig3}. These results are consistent with those reported in the literature \cite{WB02}.

\begin{figure}[h!]
\includegraphics[width = .24\textwidth, height=.22\textwidth,viewport=5  0 550 390, clip]{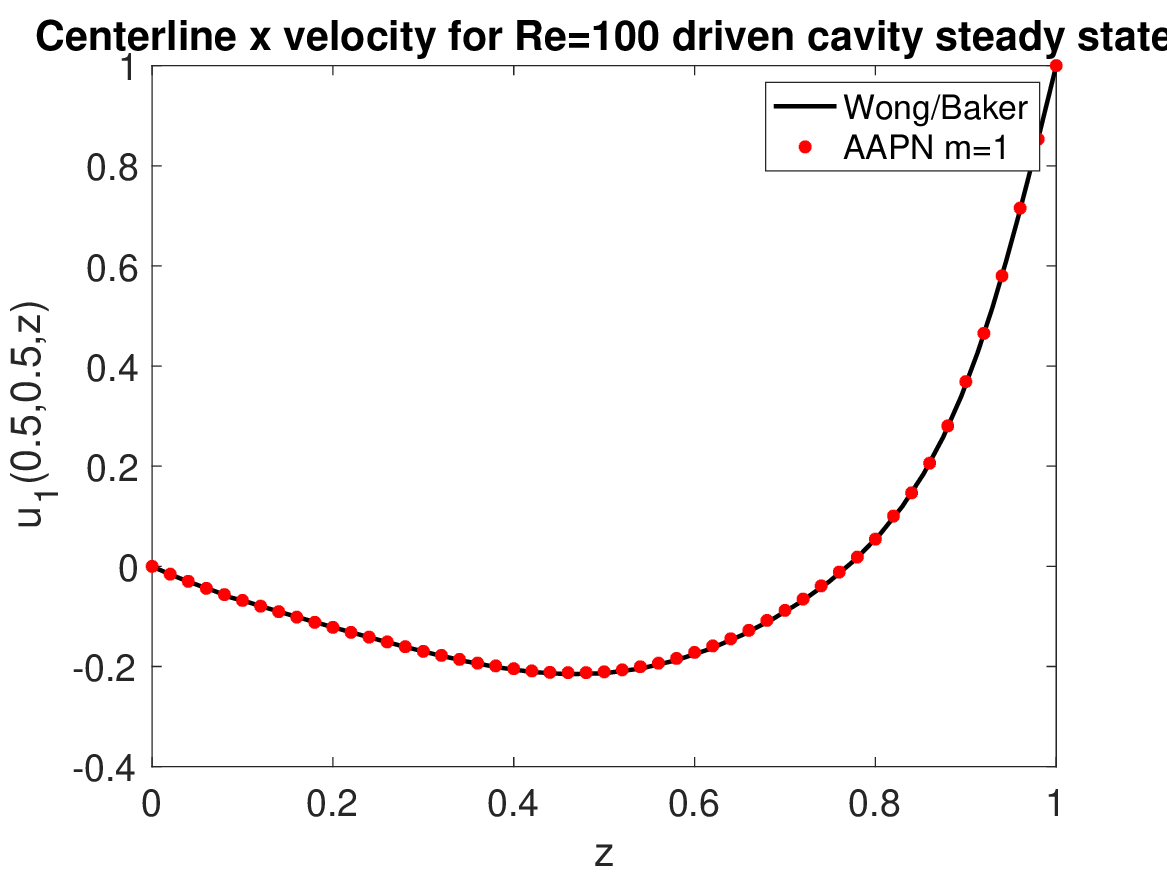}
\includegraphics[width = .74\textwidth, height=.24\textwidth,viewport=45 130 510 285, clip]{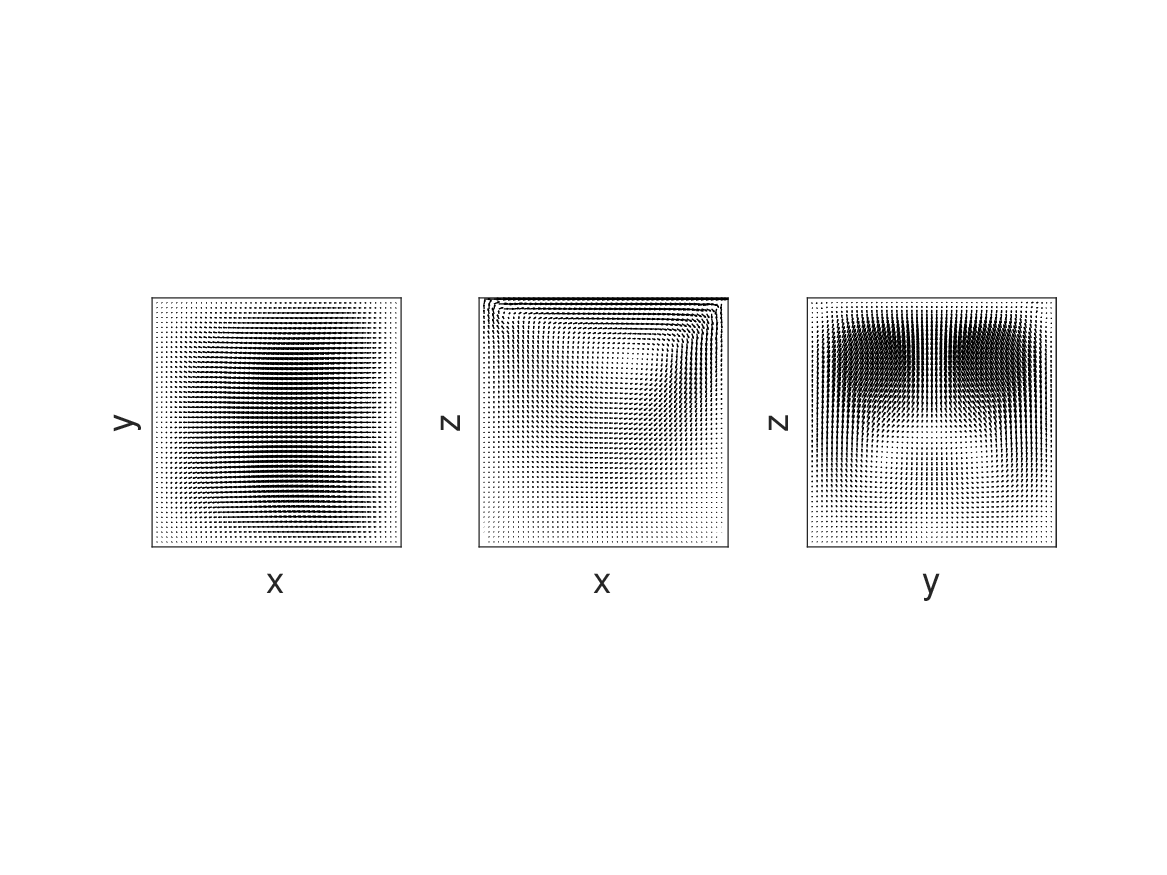}
\\
\includegraphics[width = .24\textwidth, height=.22\textwidth,viewport=5  0 550 390, clip]{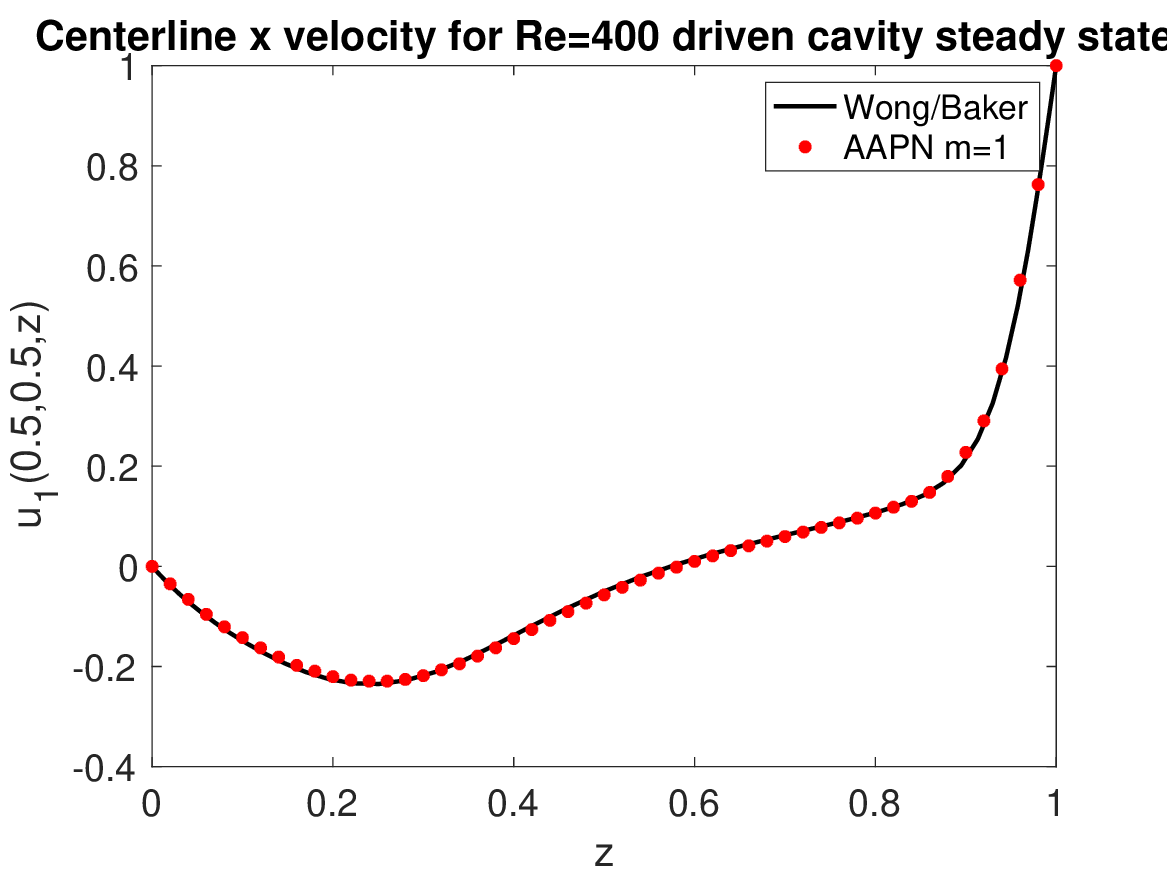}
\includegraphics[width = .74\textwidth, height=.24\textwidth,viewport=45 130 510 285, clip]{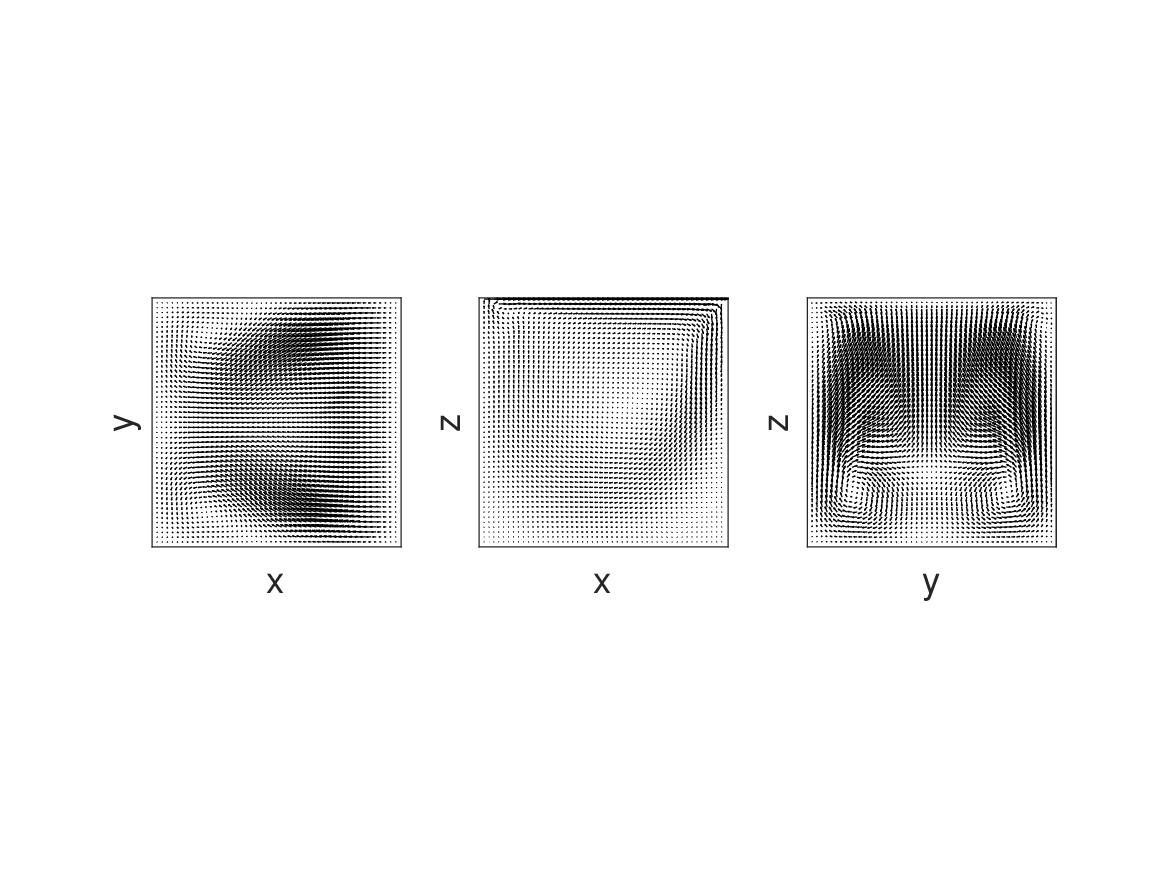}
\\
\includegraphics[width = .24\textwidth, height=.22\textwidth,viewport=5  0 550 390, clip]{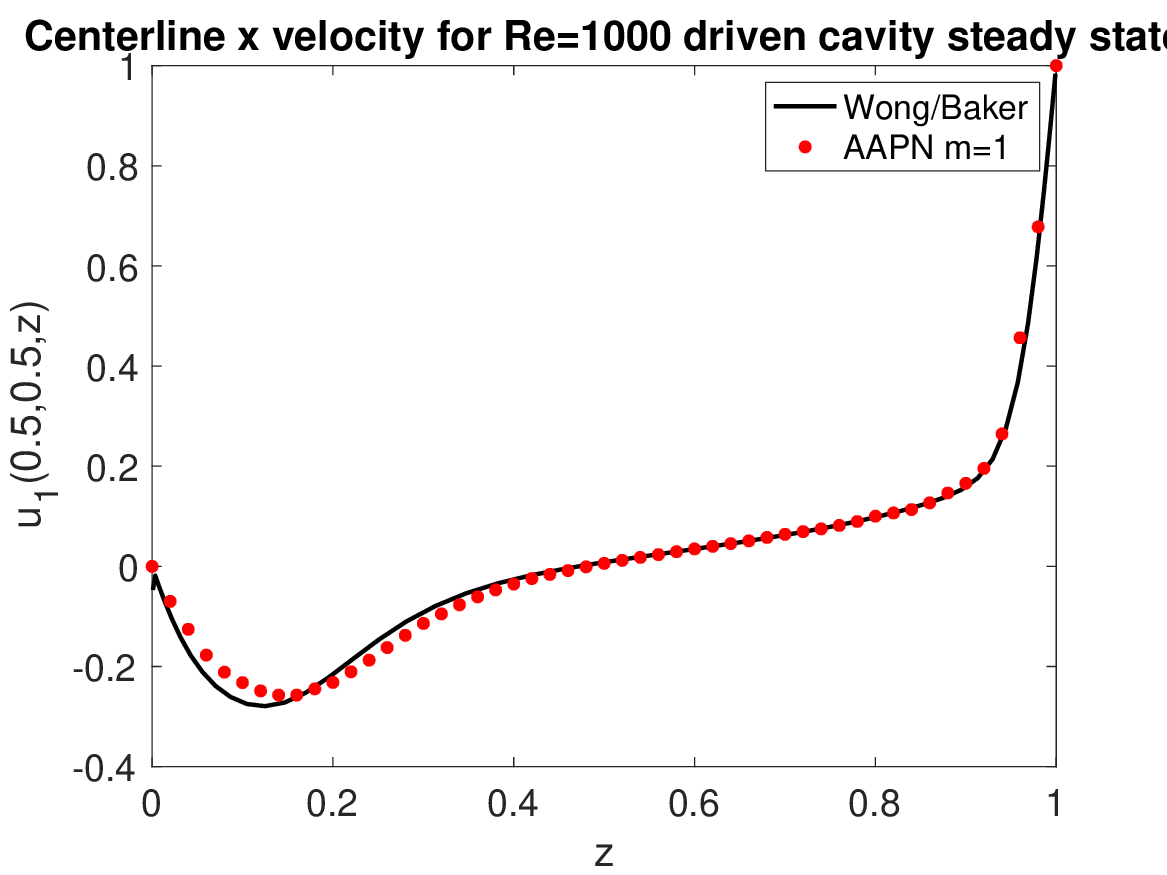}
\includegraphics[width = .74\textwidth, height=.24\textwidth,viewport=45 130 510 285, clip]{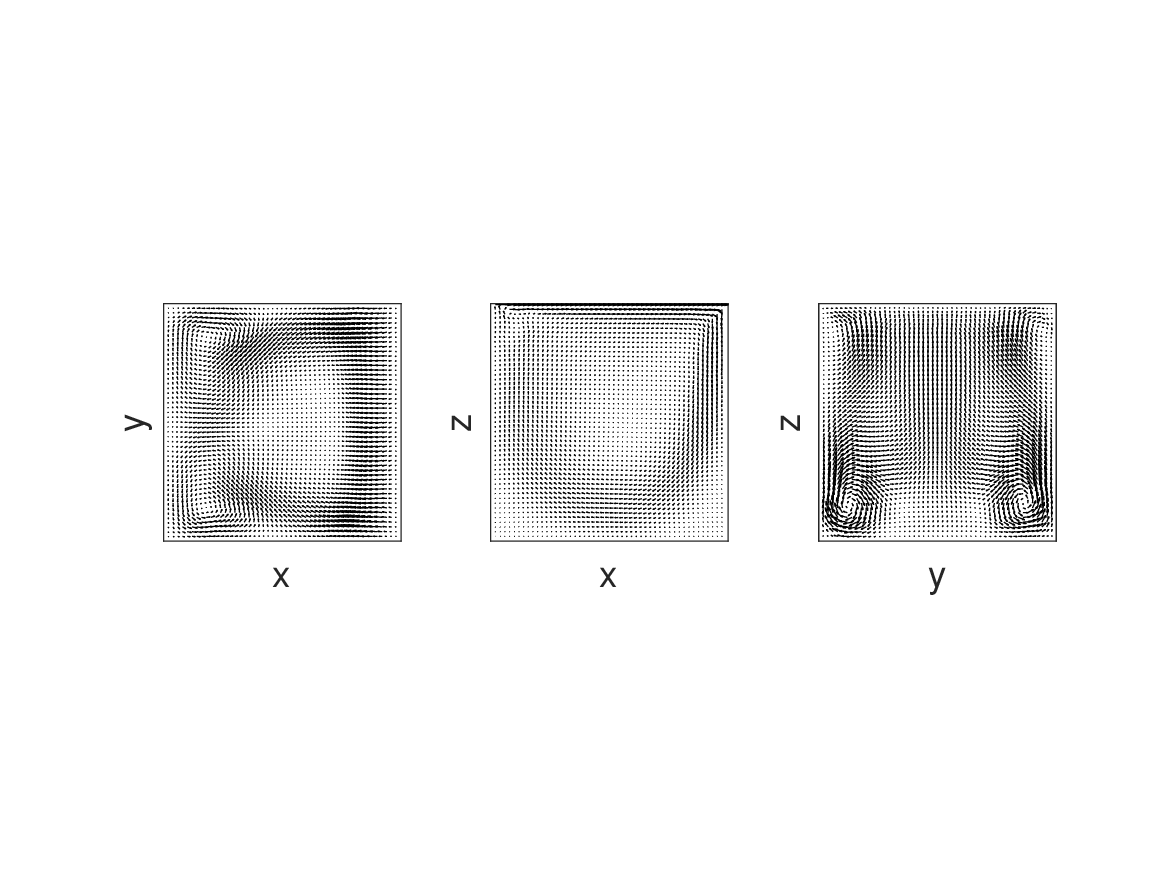}
\caption{\label{fig3}
Shown above are the centerline $x$-velocity and midsliceplanes of the solution for the 3D lid-driven cavity simulations for $Re = 100$ (top), $400$ (middle), and $1000$ (bottom), solved by the AAPicard-Newton method with depth $m=1$.
}
\end{figure}

Next, we assess the number of iterations required by the AAPicard-Newton method with varying depth $m$  for the residual to fall below $10^{-8}$ in the $H^1$ norm.  The results are presented in Table \ref{table1}. We observe that the AAPicard-Newton method significantly reduces the number of iterations needed for high $Re$ compared to the Picard-Newton method. Additionally, a larger depth proves more effective for high Reynolds numbers. 
Table \ref{table:comptime3D} presents the computational time solved by the Anderson step (Step 2), Picard step (Step 1), and Newton step (Step 3) in Algorithm \ref{alg:aapn1}, \ref{alg:aapn2}, and \ref{alg:aapnm}, under fixed $\beta =1$ and various Reynolds number $Re$ and Anderson depth $m$. We again observe that the Anderson step solve time increases for large depth $m$, but it is still significantly cheaper than both the Picard solve and the Newton solve. Thus, we recommend using a large depth $m$ for high Reynolds number problems as the AAPicard-Newton method can dramatically cut the required number of iterations.

A summary of convergence plots for large $Re$ is displayed in Figure \ref{fig4}. It is evident that the AAPicard-Newton method exhibits quadratic convergence and markedly improves convergence for large $Re$.  Furthermore, setting $\beta_{k+1} \equiv 1$ optimizes the performance of Anderson acceleration. All findings align well with Theorem \ref{thm:aapnm}.

\begin{table}[h!]
\centering
\begin{tabular}{c||c|c|c|| ccccc}
\multirow{2}*{$Re$/ Method} & \multirow{2}*{Pic.} & \multirow{2}*{Newt.} & \multirow{2}*{P-N} & \multicolumn{5}{|c}{AAP-N}\\
 &&&& $m=1$ & $m=2$ & $m=5$ & $m=10$ & $m=20$  \\ \hline
 100 &  21&5 & 4 & 4 &- &-&-&-\\
 400 &F  &8 & 6 & 6 &6& -& -&- \\
 1000 & F &B & 16 & 10 & 10& 10 & - &- \\
 1500 &  F & B &  26 & 12 & 12& 13&13 &-\\
 2000 & F &B & F &  16 & 13 &14&14&-\\
 2500 & F&B & F & F & 27 & 55 &44&51 \\
 3000 &  F&B &F & F & F & 76& 81 &39
\end{tabular}
\caption{\label{table1} 
Shown above are the convergence results (number of iterations, `F' if no convergence after 100 iterations, `B' if $H^1$ residual grows above $10^3$) for the Picard, Newton, Picard-Newton, and AAPicard-Newton methods for varying $Re$.}
\end{table}

\begin{table}[h!]
\centering
\begin{tabular}{c|| ccccc || c|c}
& \multicolumn{7}{c}{Linear Solve Time (in sec)}\\ \hline
& \multicolumn{5}{c}{ Step 2}  &Step 1 & Step 3 \\
$Re$ &$m=1$ & $m=2$ & $m=5$ & $m=10$ & $m=20$ &  Picard solve & Newton solve \\ \hline\hline
$100$  & 0.013 & - & - & - & - & 194.6 & 366.6 \\
$400$  & 0.012 & 0.013 & - & - & - & 189.4 & 363.9 \\
$1000$ & 0.010 & 0.012 & 0.021 & - & - & 185.0 & 349.6 \\
$1500$ & 0.009 & 0.012 & 0.021 & 0.035 & - & 181.6 & 346.1 \\
$2000$ & 0.010 & 0.012 & 0.021 & 0.034 & - & 181.2 & 354.9 \\
$2500$ & 0.010(F) & 0.012 & 0.021 & 0.033 & 0.064 & 181.7 & 356.1 \\
$3000$ & 0.009(F) & 0.013(F) & 0.022 & 0.034 & 0.063 & 182.4 & 346.6 \\
\end{tabular}
\caption{\label{table:comptime3D}
Shown above are the linear solve time in second (Step 2 Anderson optimization solve, Step 1 Picard solve, and Step 3 Newton solve) in Algorithm \ref{alg:aapn1}, \ref{alg:aapn2} and \ref{alg:aapnm}, with various $m$ and $Re$. `F' means no convergence after 100 iterations.}
\end{table}
 
\begin{figure}[h!]
\centering
\includegraphics[width = .44\textwidth, height=.42\textwidth,viewport=5  0 550 420, clip]{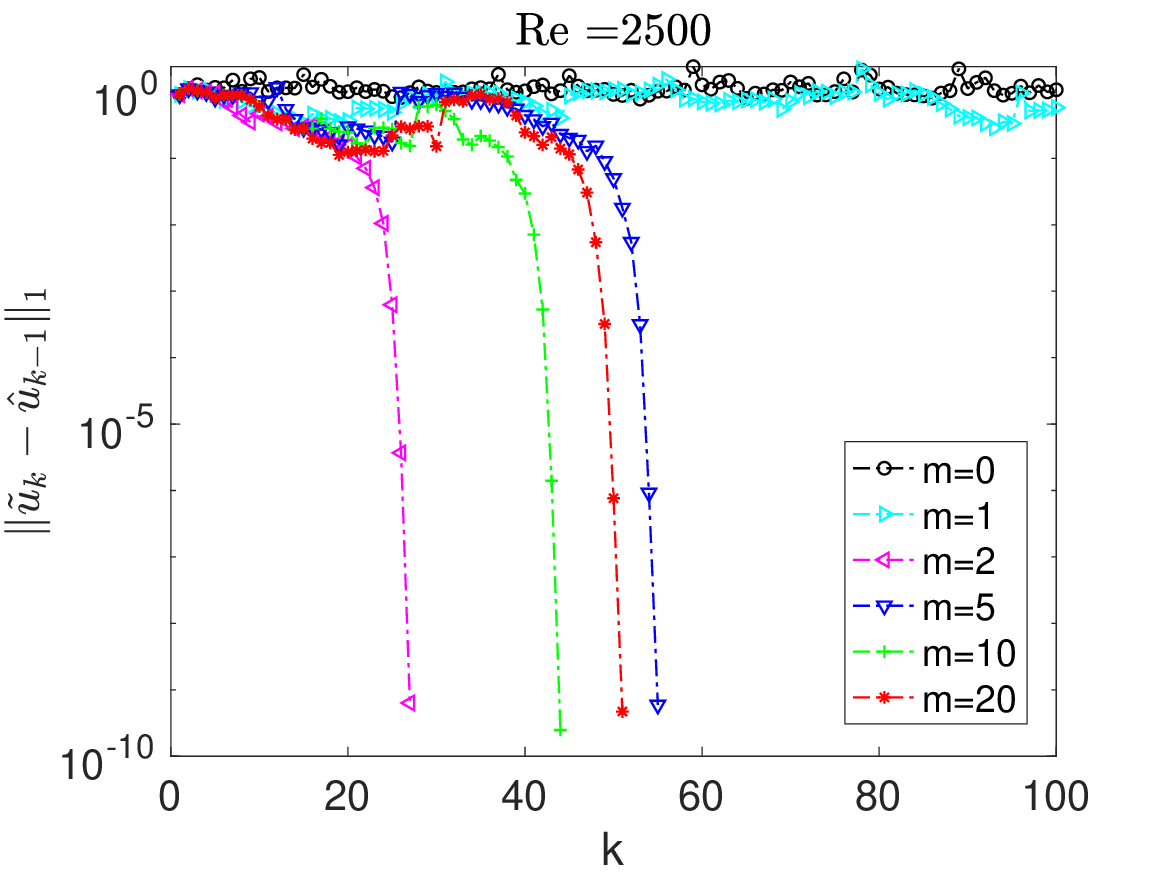}
\includegraphics[width = .44\textwidth, height=.42\textwidth,viewport=5  0 550 420, clip]{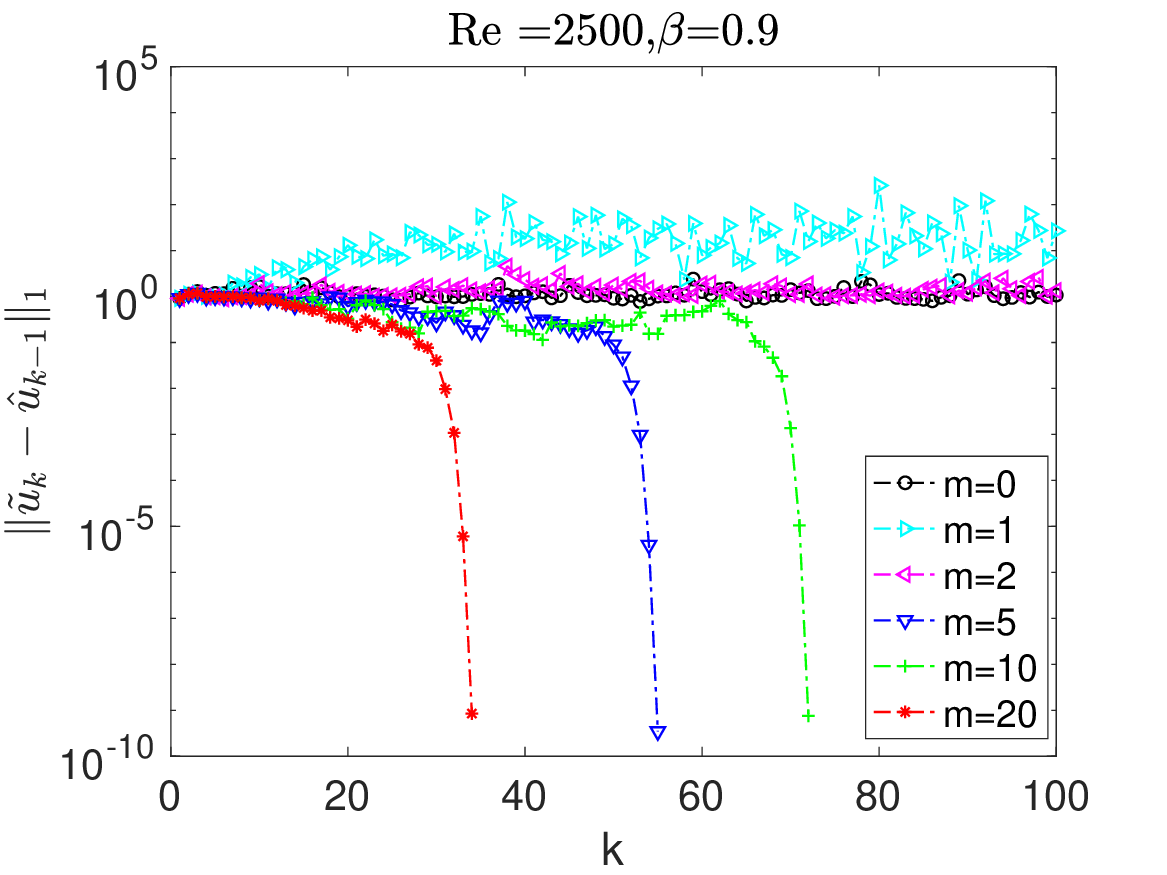}\\
\includegraphics[width = .44\textwidth, height=.42\textwidth,viewport=5  0 550 420, clip]{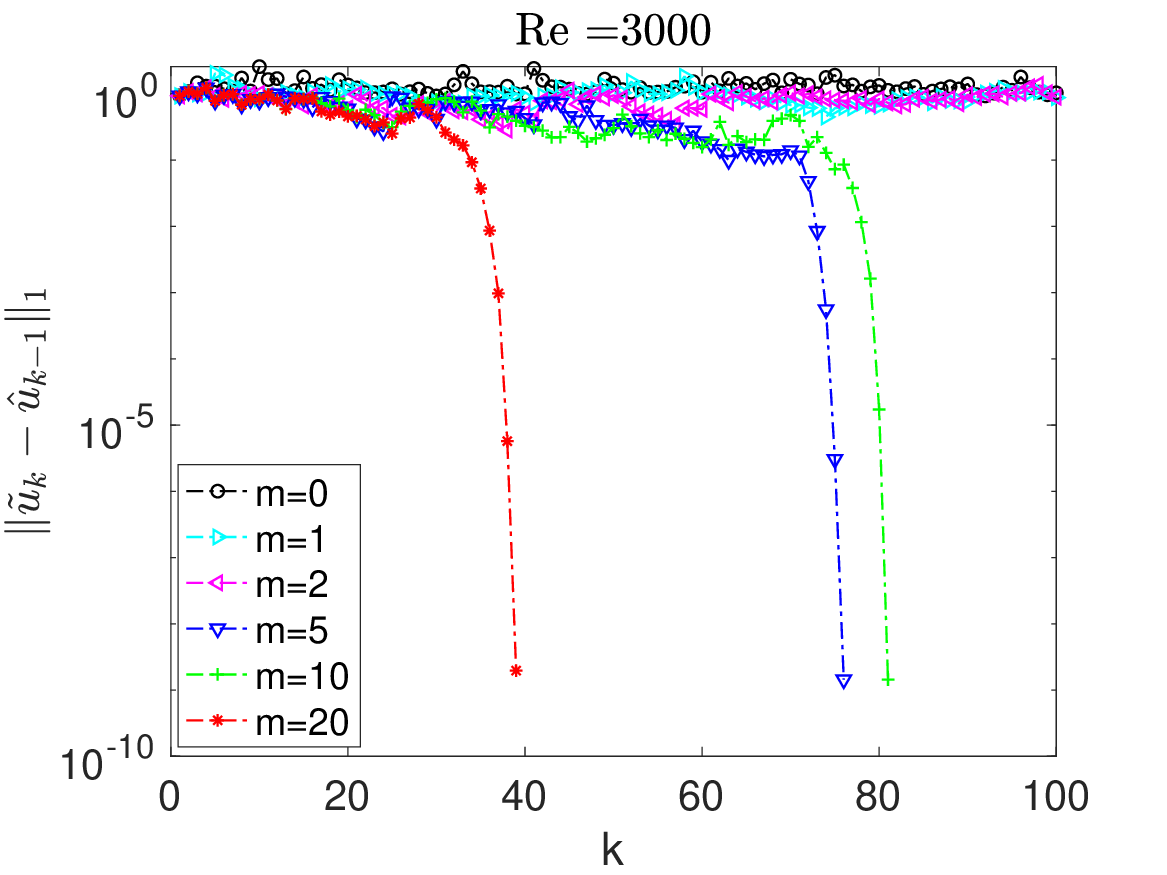}
\includegraphics[width = .44\textwidth, height=.42\textwidth,viewport=5  0 550 420, clip]{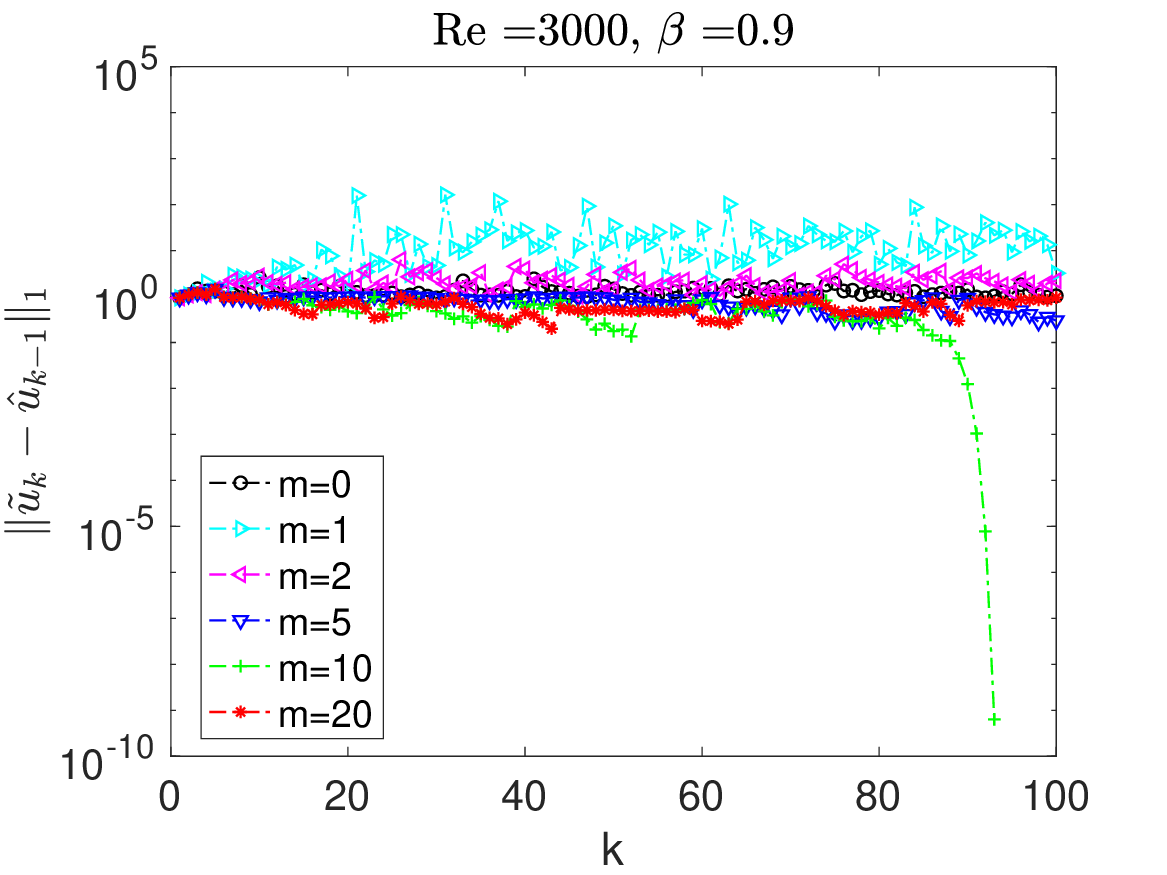}
\caption{\label{fig4}
Shown above are the convergence plots of the 3D lid-driven cavity solved by the AAPicard-Newton method with various $m = 0, 1, 2, 5,10$, and $20$, $Re = 2500$ and $3000$, and relaxation parameter $\beta =1, 0.9$.}
\end{figure}

\subsection{2D channel flow past a cylinder}
Another benchmark problem -- flow past a cylinder is presented here. The domain is a $2.2\times .41 $ rectangle, with a cylinder of radius 0.05 centered at $(0.2,0.2)$ from the bottom left corner of the rectangle. We use Scott-Vogelius elements $(P_2, P_1^{disc})$ on a barycenter mesh with a total of 79,463  dof, see Figure \ref{cylinder1} (top).

No-slip velocity boundary conditions are enforced on the cylinder and walls.  A parabolic profile is enforced nodally to be $u\vert_{\text{in/out}} = \displaystyle\left\langle 6y(0.41-y)/0.41^2, 0 \right\rangle^T $. The initial guess is set to be zero
in the interior and satisfies the boundary conditions. We solve the problem by the AAPicard-Newton method with $\beta_{k+1} \equiv 1, m=20$, and plot the contour and magnitude of the velocity field in Figure \ref{cylinder1} (bottom) for $Re=2500$.  We observe that our plots agree well with the time-averaged streamline in \cite{KT21, LSRT01, LER14}.

We also test the problem with other methods for comparison, such as Newton, Picard,  AAPicard-Newton, and Picard-Newton. Although the AAPicard-Newton method ($m=20$) uses three linear solves (one Picard solve, one optimization solve, and one Newton solve) in each iteration, we observe that it takes 15 iterations (45 linear solves) to reach the tolerance $10^{-10}$. Whereas the Picard-Newton method requires 118 iterations (236 linear solves), the Anderson accelerated Picard method with depth 20 takes 118 iterations (118 linear solves) to converge. However, the Picard and Newton methods do not converge within 150 iterations. It is worth using the AAPicard-Newton method, though it takes three linear solves at each iteration, because it significantly reduces the required number of iterations.
Thus, we conclude the AAPicard-Newton method outperforms the other methods for complicated problems (either complex geometry or high Reynolds number, or both), see Figure \ref{cylinder3}.

\begin{figure}[h!]
\centering
\includegraphics[scale=0.4,viewport=130 90 1200 300, clip]{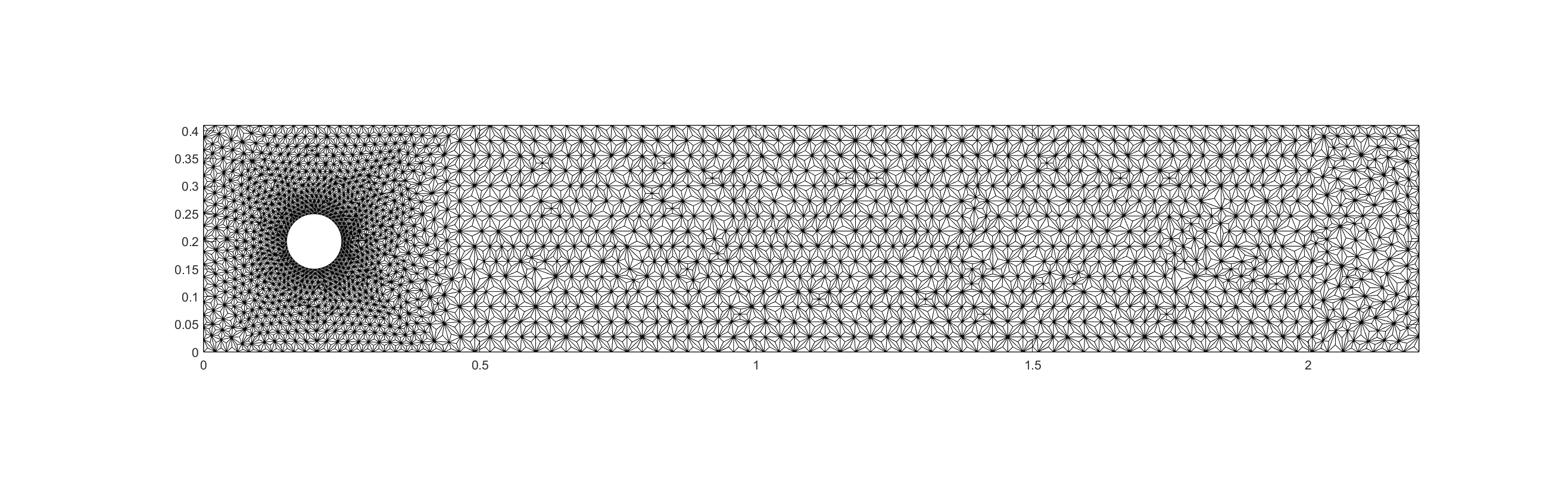}
\includegraphics[scale=0.42,viewport=70 50 1800 320, clip]{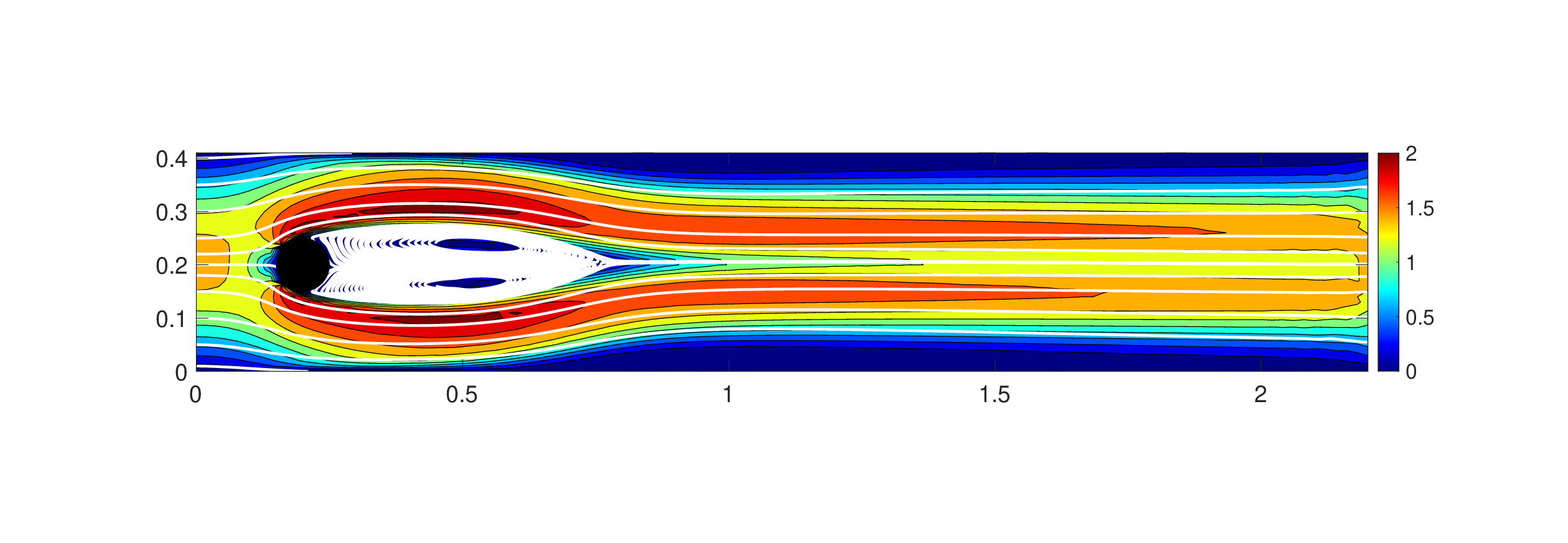}
\caption{\label{cylinder1}
Shown above are an adaptive barycenter mesh (top) and the contour level and magnitude of velocity (bottom) for flow past a cylinder test with $Re = 2500$ solved by the AAPicard-Newton method.}
\end{figure}

\begin{figure}[h!]
\centering
\includegraphics[width = .44\textwidth, height=.42\textwidth,viewport=5  0 550 420, clip]{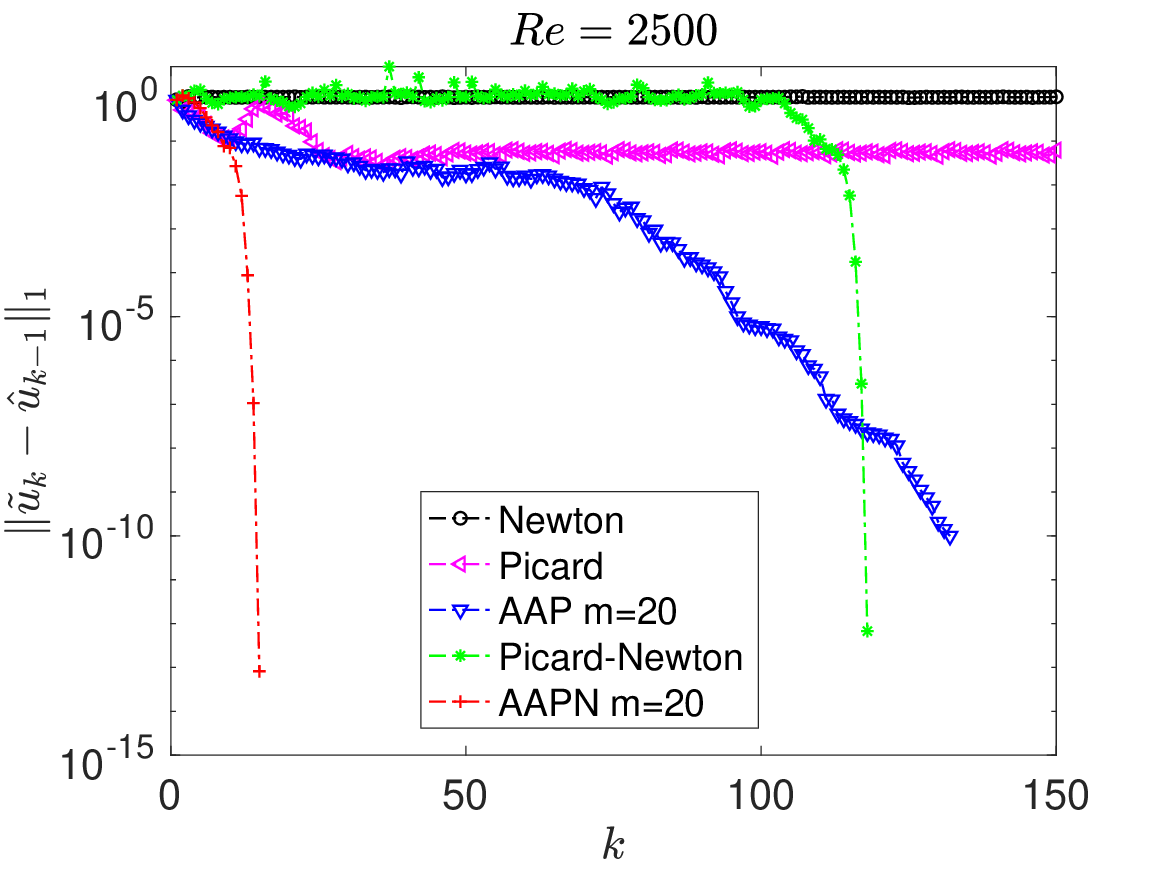}
\caption{\label{cylinder3}
Shown above is the convergence plot for flow past a cylinder test with $Re = 2500$ solved by Newton, Picard, Anderson accelerated Picard (AAP) with depth $m=20$, Picard-Newton method, and the AAPicard-Newton method (AAPN) with depth $m=20$.
}
\end{figure}

\section{Conclusions and future directions}
We proposed an easy-to-implement nonlinear preconditioning technique for Newton's method, aimed at solving the steady Navier-Stokes equations. The AAPicard-Newton method incorporates an Anderson accelerated Picard step in each iteration of Newton's method. This approach maintains quadratic convergence while ensuring global stability when the Anderson relaxation parameter $\beta_{k+1}\equiv  1$,  and a reduced convergence rate when $Re$ is sufficiently large.
Although it takes three linear solves (one Picard solve, one linear solve for optimization, and one Newton solve) in each iteration, it significantly reduces the number of required convergent iterations for higher Reynolds numbers.  
Several benchmark numerical tests demonstrate that this method offers  a much larger domain of convergence compared to the usual Newton’s method. For example, the 2D lid-driven cavity problem converges for $Re \le 20,000$ on a uniform barycenter mesh with 172.5K total dof. This performance surpasses that of the Picard method ($Re \le 4,000$), Newton's method ($Re \le 2,500$), and the Picard-Newton method ($Re \le 12,000$). In the case of the 3D lid-driven cavity, we observe convergence for $Re \le 3,000$ on a uniform barycenter mesh with 796K dof, which is still better than the Picard method ($Re \le 200$), Newton's method ($Re \le 400$), and the Picard-Newton method ($Re \le 1800$).
It is important to note that while increasing the depth $ m $ may lead to longer computational times in the linear solve for the Anderson step (Step 2 of Algorithm \ref{alg:aapnm}), it significantly reduces the number of iterations needed to satisfy the tolerance. When we compare the solve times, Anderson step is significantly cheaper than the Picard step (Step 1) and Newton step (Step 3), see Table \ref{table:comptime} and Table \ref{table:comptime3D}. Therefore, we highly recommend the AAPicard-Newton method with a depth of $m=20$ or so as an effective and optimal choice,  particularly for high Reynolds number problems, or complex geometries.
In the future, we plan to apply the AAPicard-Newton method to other fluid models that involve more complex nonlinearity, such as Bingham's problem and Boussinesq models. We will also explore whether the mesh size affects the behavior of AAPicard-Newton methods.

\bibliographystyle{plain}
\bibliography{graddiv}

\end{document}